\documentclass[11pt,reqno]{amsart}
\usepackage{amsmath,amssymb,amsfonts,amsthm,amscd,amstext,amsxtra,amsopn,array,url,verbatim,mathrsfs,enumerate,anysize}
\usepackage{graphicx}
\usepackage{amsmath,amssymb,amsfonts,amsthm,amssymb,amscd,url,amstext,amsxtra,amsopn
,txfonts}
\usepackage{verbatim}
\marginsize{2.25cm}{2.25cm}{2.25cm}{2.25cm}
\usepackage[dvipsnames,usenames]{xcolor}
\usepackage[colorlinks=true,urlcolor=Black,citecolor=Black,linkcolor=Black]{hyperref}
\makeatletter
\@namedef{subjclassname@2010}{%
  \textup{2010} Mathematics Subject Classification}
\makeatother

\usepackage{url}
\usepackage{amsthm}
\usepackage{amsmath}
\usepackage{amssymb}

\theoremstyle{plain}
\newtheorem{theorem}{Theorem}[section]
\newtheorem{lemma}[theorem]{Lemma}

\newtheorem{corollary}[theorem]{Corollary}

\theoremstyle{definition}

\newtheorem{remarks}[theorem]{Remarks}
\newtheorem{notation}[theorem]{Notation}

\numberwithin{equation}{section}

\newcommand{\li}{\text{\textnormal{li}}}
\newcommand{\aut}{\text{\textnormal{Aut}}}

\begin{document}

\title{On invariants of elliptic curves on average}

\author{Amir Akbary}
\address{University of Lethbridge, Department of Mathematics and Computer Science, 4401 University Drive, Lethbridge, AB, T1K 3M4, Canada}
\email{amir.akbary@uleth.ca}
\author{Adam Tyler Felix}
\address{University of Lethbridge, Department of Mathematics and Computer Science, 4401 University Drive, Lethbridge, AB, T1K 3M4, Canada}
\email{adam.felix@uleth.ca}

\subjclass[2010]{11G05, 11G20}

\thanks{Research of the first author is partially supported by NSERC. Research of the second author is supported by a PIMS postdoctoral fellowship. }

\keywords{\noindent reduction mod $p$ of elliptic curves, invariants of elliptic curves, average results }

\date{\today}

\begin{abstract}

We prove several results regarding some invariants of elliptic curves on average over the family of all elliptic curves inside a box of sides $A$ and $B$. As an example, let $E$ be an elliptic curve defined over $\mathbb{Q}$ and  $p$ be a prime of good reduction for $E$. Let $e_E(p)$ be the exponent of the group of rational points of the reduction modulo $p$ of $E$ over the finite field $\mathbb{F}_p$. Let $\mathcal{C}$ be the family of elliptic curves $$E_{a,b}:~y^2=x^3+ax+b,$$ where $|a|\leq A$ and $|b|\leq B$. We prove that, for any $c>1$ and $k\in \mathbb{N}$,
\begin{equation*}
\frac{1}{|\mathcal{C}|} \sum_{E\in \mathcal{C}} \sum_{p\leq x} e_E^k(p) = C_k {\rm li}(x^{k+1})+O\left(\frac{x^{k+1}}{(\log{x})^c} \right),
\end{equation*}
as $x\rightarrow \infty$, as long as  $A, B>\exp\left(c_1 (\log{x})^{1/2} \right)$ and $AB>x(\log{x})^{4+2c}$, where $c_1$ is a suitable positive constant.
Here $C_k$ is an explicit constant given in the paper which depends only on $k$, and ${\rm li}(x)=\int_{2}^x dt/\log{t}$. We prove several similar results as corollaries to a general theorem. The method of the proof is capable of improving some of the known results with $A, B>x^\epsilon$ and $AB>x(\log{x})^\delta$ to $A, B>\exp\left(c_1 (\log{x})^{1/2} \right)$ and $AB>x(\log{x})^\delta$.
\end{abstract}
\maketitle
\vspace{-1cm}
\section{INTRODUCTION AND RESULTS}

Let $E$ be an elliptic curve defined over $\mathbb{Q}$ of conductor $N$. For a prime $p$ of good reduction (i.e. $p\nmid N$), let $E_p$ be the reduction mod $p$ of $E$. It is known that $E_p(\mathbb{F}_p)$, the group of rational points of $E$ over the finite field $\mathbb{F}_p$, is the product of at most two cyclic groups, namely
$$E_p(\mathbb{F}_p)\simeq (\mathbb{Z}/i_E(p)\mathbb{Z}) \times (\mathbb{Z}/e_E(p)\mathbb{Z}),$$
where $i_{E}(p)$ divides $e_E(p)$. Thus, $e_E(p)$ is the exponent of $E_p(\mathbb{F}_p)$ and $i_E(p)$ is the index of the largest cyclic subgroup of $E_p(\mathbb{F}_p)$. In recent years there has been a lot of interest in studying the distribution of the invariants $i_E(p)$ and $e_E(p)$.  

Borosh, Moreno, and Porta \cite{BMP} were the first to study computationally $i_E(p)$ and conjectured that, for some elliptic curves, $i_E(p)=1$ occurs often. We note that $i_E(p)=1$ if and only if $E_p(\mathbb{F}_p)$ is cyclic. Let 
\begin{equation}
\label{cyclicity}
N_E(x)=\#\{p\leq x;~p\nmid N~{\rm and}~E_p(\mathbb{F}_p)~ {\rm is~cyclic}\}. 
\end{equation}
Then Serre \cite{S}, under the assumption of the generalized Riemann hypothesis (GRH) for division fields $\mathbb{Q}(E[k])$, proved that $N_E(x)\sim c_E {\rm li}(x)$ as $x\rightarrow \infty$, where $c_E>0$ if and only if 
$\mathbb{Q}(E[2])\neq \mathbb{Q}$. Here ${\rm li}(x)= \int_{2}^{x} {dt}/{\log{t}}$. For the curves with complex multiplication (CM), Murty \cite{M} removed the assumption of the GRH. Also, he showed that under GRH one can obtain the estimate $O(x\log\log{x}/(\log{x})^2)$ for the error term in the asymptotic formula for $N_E(x)$ for any elliptic curve $E$. The value of the error term is improved to $O(x^{5/6} (\log{x})^{2/3})$ in \cite{CM}. In \cite{AM2}, following the method of \cite{M} in the CM case,  the error term $O(x/(\log{x})^A)$ for any $A>1$ is established.

Another problem closely related to cyclicity is finding the average value of the number of divisors of $i_E(p)$ as $p$ varies over primes. Let $\tau(n)$ denote the number of divisors of $n$. In \cite{AG}, Akbary and Ghioca proved that $$\sum_{p\leq x} \tau(i_E(p))=c_E {\rm li}(x)+O\left(x^{5/6} (\log{x})^{2/3}\right)$$
if GRH holds, and   
$$\sum_{p\leq x} \tau(i_E(p))=c_E {\rm li}(x)+O\left(\frac{x}{(\log{x})^A}\right),$$
for $A>1$, if $E$ has CM. In the above asymptotic formulas $c_E$ is a positive constant which depends only on $E$. 

A more challenging problem is studying the average value of $i_E(p)$.  In \cite{K}, Kowalski proposed this problem  and proved unconditionally that the lower bound $\log\log{x}$ holds for $$\frac{1}{x/\log{x}} \sum_{p\leq x} i_E(p)$$
if $E$ has CM. He also showed that for a non-CM curve the above quantity is bounded from the below.

A more approachable problem is finding the average value of $e_E(p)$. Freiberg and Kurlberg \cite{FK} were the first to consider this problem and established conditional (unconditional in CM case) asymptotic formulas for $\sum_{p\leq x} e_E(p)$. The best result to date is due to Felix and  Murty \cite{FM} who proved more generally that for $k$ a fixed positive integer the following asymptotic formula holds:
$$\sum_{p\leq x} e_E^k(p)=c_{E, k} {\rm li} (x^{k+1})+O\left(x^k \mathcal{E}(x)\right),$$
where
\begin{equation*}
\mathcal{E}(x)=\begin{cases}
x/ (\log{x})^A \qquad &\text{if }E \text{ has CM}\\
x^{5/6} (\log{x})^2 &\text{if GRH holds}
\end{cases}
\end{equation*}
and $c_{E, k}$ is a positive constant depending on $E$ and $k$.
Felix and Murty derived their result as a consequence of a more general theorem on asymptotic distribution of $i_E(p)$'s. Their general theorem also imply the best known results on the cyclicity, the Titchmarsh divisor problem, and several other similar problems. 
To state their result, 
let $g(n)$ be an arithmetic function such that
\begin{equation}
\label{bound}
\sum_{n\leq x} |g(n)| \ll x^{1+\beta} (\log{x})^\gamma,
\end{equation}
where $\beta$ and $\gamma$ are arbitrary, and let 
\begin{equation}
\label{summatory}
f(n)=\sum_{d\mid n} g(d).  
\end{equation}
Then the following is proved in \cite[Theorem 1.1(c)]{FM}.
\begin{theorem}[{\bf Felix and Murty}]
\label{FMtheorem}
Under the assumption of GRH and bound \eqref{bound} for $\beta<1/2$ and arbitrary $\gamma$, we have $$\sum_{p\leq x} f(i_E(p))=c_E(f) {\rm li}(x)+O\left( x^{\frac{5+2\beta}{6}} (\log{x})^{\frac{(2-\beta)(1+\gamma)}{3}}\right),$$
where $c_E(f)$ is a constant depending only on $E$ and $f$. 
\end{theorem}
\noindent They also proved an unconditional version of the above theorem for CM elliptic curves (see \cite[Theorem 1.1(a)]{FM}). 

Our goal in this paper is to prove that Theorem \ref{FMtheorem} holds unconditionally on average over the family of all elliptic curves in a box. More precisely, we consider the family $\mathcal{C}$ of elliptic curves $$E_{a,b}: y^2=x^3+ax+b,$$ where $|a|\leq A$ and $|b|\leq B$. It is not that difficult to prove a version of Theorem \ref{FMtheorem} on average over a large box. However it is a challenging problem to establish the same over a thin box. By a \emph{thin} box we mean, as a function of $x$, either $A$ or $B$ can be as small as $x^\epsilon$ for any $\epsilon>0$. Here we prove a stronger result in which one of $A$ and $B$ can be as small as $\exp(c_1(\log{x})^{1/2})$ for a suitably chosen constant $c_1>0$. Before stating our main theorem, 
we note that, at the expense of replacing $\beta$ and $\gamma$ by larger non-negative values, we can assume that $\beta$ and $\gamma$ are non-negative.
\begin{theorem}
\label{theorem1}
 Let $c>1$ be a positive constant and let $f$ be the summatory function \eqref{summatory} of a function $g$ that satisfies \eqref{bound} for certain non-negative values of $\beta$ and $\gamma$. 
 Assume that $AB> x(\log{x})^{4+2c}$ if $0\leq\beta<1/2$ and $AB> x^{1/2+\beta}(\log{x})^{2\gamma+6+2c}(\log\log{x})^2$ if $1/2\leq \beta<1$. Then there is a positive constant $c_1>0$ such that if  $A,B>\exp\left( c_1 (\log{x})^{1/2} \right)$, we have
$$\frac{1}{|\mathcal{C}|}\sum_{E_{a, b}\in \mathcal{C}} \sum_{p\leq x} f(i_{E_{a, b}}(p)) =c_0(f) {\rm li}(x)+O\left( \frac{x}{(\log{x})^{c}} \right),$$
where 
\begin{equation}
 \label{notsure}
c_0(f):=\sum_{d \ge 1} \frac{g(d)}{d\psi(d)\varphi(d)^2}.
\end{equation}
The implied constant depends on $g$, $\beta$, $\gamma$, and $c$.   Here $\varphi(n)=n\prod_{d\mid n} (1-1/p)$ and $\psi(n)=\prod_{d\mid n} (1+1/p)$.
\end{theorem}

This theorem is comparable to Stephens's average result on Artin's primitive root conjecture. Let $a$ be a non-zero integer other than $-1$ or a perfect square and let $A_a(x)$ be the number of primes not exceeding $x$, for which $a$ is a primitive root. The following result has been proved in \cite{Stephens-I} and \cite{Stephens-II}.

\begin{theorem}[{\bf Stephens}]
\label{Stheorem}
There exist a constant $c_1>0$ such that, if $N> \exp\left(c_1(\log{x})^{1/2} \right)$, then
$$\frac{1}{N} \sum_{a\leq N} A_a(x)=A~ {\rm li}(x)+O\left( \frac{x}{(\log{x})^c} \right),$$
where $A=\prod\limits_{\ell\textnormal{ prime}}(1-1/\ell(\ell-1))$ and $c$ is an arbitrary constant greater than $1$.
\end{theorem}

The line of research on Artin primitive root conjecture on average started with the work of Goldfeld \cite{G} that used multiplicative character sums and the large sieve inequality to establish a weaker version of Theorem \ref{Stheorem}. The extension of the method of character sums to the average questions on a two parameters family, in the case of elliptic curves inside a box, was pioneered by Fouvry and Murty in \cite{FM1} on the average Lang-Trotter conjecture for supersingular primes. Their work was extended to the general Lang-Trotter conjecture by David and Pappalardi \cite{DP}.  The best result on the size of the box ($|a|\leq A$ and $|b|\leq B$) is due to Baier \cite{B1} who established the Lang-Trotter conjecture on average under the condition 
\begin{equation}
\label{bound1}
A, B>x^{1/2+\epsilon}~~{\rm and}~~AB>x^{3/2+\epsilon},
\end{equation}
where $\epsilon>0$. The supersingular case of this result is due to Fouvry and Murty \cite[Theorem 6]{FM1}.  Baier \cite{B2} has also established  an average result for the Lang-Trotter conjecture on the range 
\begin{equation}
\label{bound2}
A, B>(\log{x})^{60+\epsilon}~~{\rm and}~~x^{3/2} (\log{x})^{10+\epsilon} <AB <e^{x^{1/8-\epsilon}},
\end{equation}
where $\epsilon>0$. Note that \eqref{bound2} is superior to \eqref{bound1} if $A$ and $B$ are not very large.

There are also average results for other distribution problems for elliptic curves. Banks and Shparlinski \cite{BS} considered such average problems in a very general setting by employing multiplicative characters and consequently proved average results for the cyclicity problem, the Sato-Tate conjecture, and the divisibility problem on a box $|a|\leq A$, $|b| \leq B$ satisfying the conditions
\begin{equation}
\label{bound3}
A, B\leq  x^{1-\epsilon}~~{\rm and}~~AB\geq x^{1+\epsilon},
\end{equation}
where $\epsilon>0$. Another notable result is related to Koblitz conjecture. Let 
$$\pi_E^{\rm twin} (x):=\#\{p\leq x;~\#E_p(\mathbb{F}_p)~\textrm{is prime}\}.$$
A conjecture of Koblitz predicts that
$$\pi_E^{\rm twin} (x)\sim c_E \frac{x}{(\log{x})^2},$$
as $x\rightarrow \infty$, where $c_E$ is a constant depending on $E$. Balog, Cojocaru, and David proved the following result on Koblitz conjecture on the average over the family $\mathcal{C}$. 
\begin{theorem}[{\bf Balog, Cojocaru, and David}]
\label{twin}
Let $A, B>x^\epsilon$ and $AB>x(\log{x})^{10}$. Then, as $x\rightarrow\infty$, 
$$\frac{1}{|\mathcal{C}|} \sum_{E\in \mathcal{C}} \pi_E^{\rm twin}(x)=\prod_{{\rm prime}~\ell}\left(1-\frac{\ell^2-\ell-1}{(\ell-1)^3 (\ell +1)}\right) \frac{x}{(\log{x})^2}+O\left( \frac{x}{(\log{x})^3} \right).$$
\end{theorem}
\noindent (See \cite[Theorem 1]{BCD}.) 

The error term in the above theorem is estimated by a careful analysis of some multiplicative character sums.
We prove our Theorem \ref{theorem1} by a generalization of a modified version of \cite[Lemma 6]{BCD} (see our Lemma \ref{lemma:2}). We have used some results of Stephens \cite{Stephens-II} to sharpen the estimates given in \cite[Lemma 6]{BCD}, and thus we could establish our results, for $\beta<1/2$, on a box of size   
\begin{equation}
\label{bound4}
A, B>  \exp(c_1(\log{x})^{1/2})~~{\rm and}~~AB> x (\log{x})^\delta,
\end{equation}
for appropriate positive constants $c_1$ and $\delta$.  As far as we know this is the thinnest box used  for an elliptic curve average problem. Our Theorem \ref{theorem1} has many applications.  Here we mention some direct consequence of it to the cyclicity problem, the Titchmarsh divisor problem, and computation of the $k$-th power moment of the exponent $e_E(p)$. 
\begin{corollary}
\label{corollaries}
Let $c>1$ and $AB>x(\log{x})^{4+2c}$. There is $c_1>0$ such that if $A, B>\exp\left(c_1 (\log{x})^{1/2} \right)$ then, as $x\rightarrow \infty$, the following statements hold.
\begin{enumerate}[\upshape (i)]
\item $$\frac{1}{|\mathcal{C}|} \sum_{E\in \mathcal{C}} N_E(x) = \left( \sum_{d\geq 1} \frac{\mu(d)}{d\psi(d) \varphi(d)^2}\right) {\rm li}(x)+O\left(\frac{x}{(\log{x})^c} \right),$$where $N_E(x)$ is the cyclicity counting function and $\mu(d)$ is the M\"{o}bius function.
\item $$\frac{1}{|\mathcal{C}|} \sum_{E\in \mathcal{C}} \sum_{p\leq x} \tau(i_E(p)) = \left( \sum_{d\geq 1} \frac{1}{d\psi(d) \varphi(d)^2}\right) {\rm li}(x)+O\left(\frac{x}{(\log{x})^c} \right).$$
\item For $k\in \mathbb{N}$ we have $$\frac{1}{|\mathcal{C}|} \sum_{E\in \mathcal{C}} \sum_{p\leq x} e_E^k(p) = \left( \sum_{d\geq 1} \frac{\sum_{\delta|d}\mu(\delta)\delta^{k}}{d^{k+1}\psi(d) \varphi(d)^2}\right) {\rm li}(x^{k+1})+O\left(\frac{x^{k+1}}{(\log{x})^c} \right).$$
\end{enumerate}
\end{corollary} 

Part (i) of the above corollary gives a strengthening of a result of Bank and Shparlinski \cite[Theorem 18]{BS} where asymptotic formula in (i) was proved in the weaker range \eqref{bound3}. Parts (ii) and (iii) establish unconditional average versions of some results given in \cite{AG} and \cite{FM}.

\begin{remarks}\label{remarks}
(i) As corollaries of Theorem \ref{theorem1} we can also establish unconditional average results for $f(i_E(p))$, where $f(n)$ is one of the functions $(\log{n})^\alpha$, $\omega(n)^k$, $\Omega(n)^k$, $2^{k\omega(n)}$, or $\tau_k(n)^r$. Here $\alpha$ is an arbitrary positive real number and $k$ and $r$ are fixed non-negative integers.  See \cite[p. 276]{FM} for conditional results related to these functions in the case of a single elliptic curve.

(ii) Under the conditions of Theorem \ref{theorem1} one can also obtain average results for $f(n)=n^\beta$ and $f(n)=\sigma_\beta(n)=\sum_{m\mid n} m^\beta$ as long as $\beta<1$. More precisely, for $A$ and $B$ satisfying the conditions of Theorem \ref{theorem1} we have, for $c>1$, 
$$\frac{1}{|\mathcal{C}|} \sum_{E\in \mathcal{C}} \sum_{p\leq x} i_E^\beta(p) = \left( \sum_{d\geq 1} \frac{g(d)}{d\psi(d) \varphi(d)^2}\right) {\rm li}(x)+O\left(\frac{x}{(\log{x})^c} \right),$$
where $g$ is the unique arithmetical function satisfying $$n^\beta=\sum_{m\mid n} g(m).$$
This stops short of providing an answer on average to a problem proposed by  Kowalski \cite[Problem 3.1]{K} that asks about asymptotic behavior of $\sum _{p\leq x} i_E(p)$.

(iii) Following the proof of Theorem \ref{theorem1}, one can improve the condition $A, B>x^\epsilon$ in Theorem \ref{twin} to $A, B>\exp\left(c_1 (\log{x})^{1/2}  \right)$, for some suitably chosen constant $c_1$.

(iv) Lemma \ref{lemma:2} is the difficult part of the proof of Theorem \ref{theorem1}. The proof of Lemma \ref{lemma:2} follows the method used in the proof of Lemma 6 of \cite{BCD} (which itself is based on \cite{BS}) and combines it with some devices from \cite{Stephens-II}. A new ingredient in the proof of Lemma \ref{lemma:2} is an asymptotic estimate due to Howe (see Lemma \ref{lemma:1}) for the number of elliptic curves over $\mathbb{F}_p$ which have $d$-torsion subgroup over $\mathbb{F}_p$ isomorphic to two copies of $\mathbb{Z}/d\mathbb{Z}$. Another new feature is a successful application of Burgess's bound (see Lemma \ref{Burgess}) in handling terms obtained from the error term of Howe's estimate. 

(v) One other novel feature of the proof of Theorem \ref{theorem1} is  sharp estimates of the error terms arising from the curves of $j$-invariant $0$ or $1728$, which are estimated using some results from the theory of CM curves (see Lemma \ref{claim}). A trivial estimate of these terms will result in unsatisfactory upper bounds on admissible values  of $A$ and $B$ in Theorem \ref{theorem1}.
\end{remarks}
%

%
Following the ideas of the proof of Theorem \ref{theorem1} and by a careful analysis of some character sums one can show that $c_{0}(f)\li(x)$ closely approximates $\sum_{p \le x} f(i_{E}(p))$ for almost all curves $E\in \mathcal{C}$.  Here we prove the following more general theorem.
\begin{theorem}
\label{Theorem 2}
Let $0\leq \beta<1/2$ and $\gamma\geq 0$. Let $f(n)$ be an arithmetic function satisfying 
\begin{equation}
\label{2bound}
f(n) \ll n^\beta (\log{n})^{\gamma}.
\end{equation}
Suppose $AB> x^2(\log{x})^{6}$ if  $0\leq\beta<1/4$ and $AB> x^{\frac{3}{2}+2\beta}(\log{x})^{4\gamma+14}(\log\log{x})^4$ if $1/4\leq \beta<1/2$. Then there is a positive constant $c_1>0$ such that, if  $A,B>\exp\left( c_1 (\log{x})^{1/2} \right)$, we have
$$\frac{1}{|\mathcal{C}|} \sum_{E \in \mathcal{C}} \left(\sum_{p \le x} f(i_{E}(p))-c_{0}(f)\li(x)\right)^{2} =O\left( \frac{x^2}{(\log{x})^2}\right),$$ 
where $c_0(f)$ is defined by \eqref{notsure}.
\end{theorem}
The following is a direct consequence of Theorem \ref{Theorem 2}.
\begin{corollary}
\label{Turan}
Let $h(x)$ be a positive real function such that $\lim\limits_{x\rightarrow \infty} h(x)=0$. Under the assumptions of Theorem \ref{Theorem 2}, for any $x>1$ we have
\begin{equation}
\label{almost}
\left| \sum_{p \le x} f(i_{E}(p))-c_{0}(f)\li(x) \right|\leq \frac{x}{h(x)\log{x}},
\end{equation}
for almost all $E\in \mathcal{C}$. More precisely  \eqref{almost} holds except possibly for $O\left(h(x)^2 |\mathcal{C}|\right)$ of curves in $\mathcal{C}$.
\end{corollary}
\indent We note that one can take $f$ to be any of the functions mentioned in Corollary \ref{corollaries} (i), (ii) and Remarks \ref{remarks} (i) and (ii).  For Corollary \ref{corollaries} (i), the corresponding function to $f(n)$ is the characteristic function of the singleton set $\{1\}$.
\begin{remarks}
It is possible to establish a version of Theorem \ref{Theorem 2} using the bound 
$$ \sum_{n\leq x} |g(n)|^2 \ll x^{1+2\beta} (\log{x})^{2\gamma}$$
instead of \eqref{2bound}. However we find that \eqref{2bound} will make the presentation of the proof more convenient. Note that if $$f(n)=\sum_{d\mid n} g(d)\ll n^\beta (\log{n})^\gamma$$ then, by the M\"{o}bius inversion formula, 
we have $$\sum_{n\leq x} |g(n)|^2 \ll x^{1+2\beta} (\log{x})^{2\gamma+1}.$$
\end{remarks}

The structure of the paper is as follows. In Section 2 we summarize results that will be used in the proof of our two theorems. Section 3 is dedicated to a detailed proof of Theorem \ref{theorem1} and Corollary \ref{corollaries}.
In Section 4 we briefly summarize the proof of a technical lemma which is a two-dimensional version of Lemma \ref{lemma:2}. The proof is tedious and divides to several subcases. We treat some cases and briefly comment on the remaining ones.  Finally in Section 5 we prove Theorem \ref{Theorem 2}.

\begin{notation}
Throughout the paper $p$ and $q$ denote primes (for simplicity in most cases we assume that $p, q \neq 2, 3$), $\varphi(n)$ is the Euler function, $\omega(n)$ is the number of distinct prime divisors of $n$, $\Omega(n)$ is the total number of prime divisors of $n$, $\tau(n)$ is the total number of divisors of $n$,  $p(n)$ is the largest prime factor of $n$, $\tau_k(n)$ is the number of representations of $n$ as a product of $k$ natural numbers, $\mu(n)$ is the M\"{o}bius function, $\psi(n)=n\prod_{d\mid n} (1+1/d)$, and $\pi(x; d, a)$ is the number of primes not exceeding $x$ that are congruent to $a$ modulo $d$.
Moreover, $K$ is a quadratic imaginary number field of class number $1$,
$N(\mathfrak{a})$ is the norm of an ideal $\mathfrak{a}$ of $K$, $N(\alpha)$ is the norm of an element $\alpha$ in $K$, $\mathfrak{p}$ always denotes a degree $1$ prime ideal of $K$ with $N(\mathfrak{p})=p$, and $d_{\rm sp}$ is the largest divisor of $d$ composed of primes that split completely in $K$. We denote the finite field of $p$ elements by $\mathbb{F}_p$ and its multiplicative group by $\mathbb{F}_p^\times$. For two functions $f(x)$ and $g(x)\neq 0$, we use the notation $f(x)=O(g(x))$, or alternatively $f(x)\ll g(x)$, if $|f(x)/g(x)|$ is  bounded as $x\rightarrow \infty$.
\end{notation}

\section{LEMMAS}
Let $E_{s, t}$ denote an elliptic curve over $\mathbb{F}_p$ given by the equation
\begin{equation*}
y^2=x^3+sx+t;~~s, t\in \mathbb{F}_p,
\end{equation*}
where at least one of $s$ or $t$ is non-zero.
Let $E_{s, t}[d](\mathbb{F}_p)$ denote the set of $d$-torsion points of $E_{s, t}$ with coordinates in $\mathbb{F}_p$. 
The following lemma essentially is due to Howe (see \cite[p. 245]{Howe}).

\begin{lemma}[]\label{lemma:1}
\begin{enumerate}[\upshape (i)]
\item\label{lemma:1.i} For $d\in \mathbb{N}$ and a fixed prime $p$, let
\begin{equation*}
\mathcal{S}_d(p):=\left\{ (s, t)\in \mathbb{F}_p \times \mathbb{F}_p;~E_{s, t}[d](\mathbb{F}_p)\simeq \mathbb{Z}/d\mathbb{Z} \times \mathbb{Z}/d\mathbb{Z} \right\}.
\end{equation*}
For $d \mid p-1$,
we have
\begin{equation*}
\#\mathcal{S}_d(p)=\frac{p(p-1)}{d\psi(d) \varphi(d)}+O( p^{3/2}).
\end{equation*}
Moreover, if $d\nmid p-1$ or $d > \sqrt{p}+1$, then $\#\mathcal{S}_{d}(p)=0$.
\item\label{lemma:1.ii} The assertions in (\ref{lemma:1.i}) hold if we replace $\mathcal{S}_d(p)$ with $\tilde{\mathcal{S}}_d(p)$, where
$$\tilde{\mathcal{S}}_d(p):=\left\{ (s, t)\in \mathbb{F}_p^\times \times \mathbb{F}_p^\times;~E_{s, t}[d](\mathbb{F}_p)\simeq \mathbb{Z}/d\mathbb{Z} \times \mathbb{Z}/d\mathbb{Z} \right\}.$$
\end{enumerate}
\end{lemma}
\begin{proof}
\begin{enumerate}[\upshape (i)]
\item   We know that elliptic curves isomorphic (over $\mathbb{F}_p$) to  $E_{s, t}$ are of the form $E_{su^4, tu^6}$, where $u\in\mathbb{F}_p^{\times}$. Let $\textrm{Aut}_{\mathbb{F}_p}(E_{s, t})$ be the group of automorphisms (over $\mathbb{F}_p$) of the elliptic curve $E_{s, t}$. So the number of elliptic curves isomorphic to $E_{s, t}$ (over $\mathbb{F}_p$) is $(p-1)/|\textrm{Aut}_{\mathbb{F}_p}(E_{s, t})|$. Let $[E_{s,t}]$ denote the class of all elliptic curves over $\mathbb{F}_p$ that are isomorphic over $\mathbb{F}_p$ to $E_{s, t}$. We have
$$\#\mathcal{S}_d(p)=\sum_{[E_{s, t}]\subset\mathcal{S}_d(p)} \frac{p-1}{\left|\textrm{Aut}_{\mathbb{F}_p}(E_{s, t})\right|}.$$
Now the result follows since by \cite[p. 245]{Howe}, we have, for $d\mid p-1$,
\begin{equation}\label{eq:1.1}
\sum_{[E_{s, t}]\subset\mathcal{S}_d(p)} \frac{1}{\left|\textrm{Aut}_{\mathbb{F}_p}(E_{s, t})\right|}=\frac{p}{d\psi(d) \varphi(d)}+O( p^{1/2}).
\end{equation}
Moreover, by \cite[Corollary III.8.1.1]{Silverman}, if $d\nmid p-1$ then $(\mathbb{Z}/d\mathbb{Z})^{2} \not\cong E_{s,t}(\mathbb{F}_{p})[d]$, and so $\#\mathcal{S}_d(p)=0$. Also if $d > \sqrt{p}+1$ and $(\mathbb{Z}/d\mathbb{Z})^{2} \cong E_{s,t}(\mathbb{F}_{p})[d] \subseteq E_{s,t}(\mathbb{F}_{p})$, then $p+2\sqrt{p}+1<d^2 \leq \#E_{s,t}(\mathbb{F}_{p})$. On the other hand $\#E_{s,t}(\mathbb{F}_{p}) \le p+2\sqrt{p}+1$, by Hasse's theorem.  This is a clear contradiction.
\item We can deduce this by following the proof of part (i) and observing that there are $O(1)$ isomorphism classes over $\mathbb{F}_p$ containing a curve of the form $E_{0,t}$ or $E_{s, 0}$.
\end{enumerate}
\end{proof}

\begin{remarks}
\label{automorphism}
\begin{enumerate}[\upshape (i)]
\item For any prime $p$, we know that $|\aut_{\mathbb{F}_{p}}(E_{s,t})| =O(1)$.  In fact, for $p \ne 2,3$, from \cite[Theorem III.10.1]{Silverman}, we know that
\begin{equation*}
|\aut_{\mathbb{F}_p}(E_{s, t})|=
\begin{cases}
6&\textrm{if}~s=0~\textrm{and}~p\equiv 1~(\textrm{mod}~6)\\
4&\textrm{if}~t=0~\textrm{and}~p\equiv 1~(\textrm{mod}~4)\\
2&\textrm{otherwise}\\
\end{cases}.
\end{equation*}
\item We note that, using Howe's notation \cite[Page 245]{Howe}, we have
\begin{equation*}
\sum_{[E_{s,t}] \subset \mathcal{S}_{d}(p)} \frac{1}{|\aut_{\mathbb{F}_{p}}(E_{s,t})|}=\frac{p}{d\psi(d)\varphi(d)}+O\left(\psi(d/d)2^{\omega(d)}\sqrt{p}\right),
\end{equation*}
where the implied constant is absolute.  However, the term $2^{\omega(d)}$ is a bound for $\sum_{j|\frac{\gcd(d,p-1)}{d}} \mu(j)$.  In our case, $\tfrac{\gcd(d,p-1)}{d}=1$, since $d\mid p-1$.
Thus, the term $2^{\omega(d)}$ can be removed.  Also, $\psi(d/d)=1$, and thus (\ref{eq:1.1}) is correct.
\end{enumerate}
\end{remarks}

Let $K$ be a quadratic imaginary number field of class number $1$. Let $\mathfrak{p}$ be a degree $1$ prime ideal of $K$ with $N(\mathfrak{p})=p$. Let $\pi_p$ be the unique generator of $\mathfrak{p}$.  Note that if $\mathfrak{p}$ is unramified, then $\pi_p$ is unique up to units, and if it is ramified, then $\pi_{p}$ is unique up to units and complex conjugate. We have $N(\mathfrak{p})=N(\pi_p)=p$. 
\begin{lemma}
\label{claim}
Suppose that $d_{\rm sp}$ is the largest divisor of $d$ composed of primes that split completely in $K$.

\begin{enumerate}[\upshape (i)]
\item For positive integer $d$ with $d^2\leq x/\log{x} $ we have
$$\sum_{\substack{{N(\mathfrak{p})\leq x}\\{d \mid (\pi_p-1)(\bar{\pi}_p-1)}}} 1 \ll \frac{2^{\omega(d_{\rm sp})}\tau(d_{\rm sp})}{\varphi(d)} \frac{x}{\log(x/d^2)}.$$
\item For positive integer $d$, we have
$$\sum_{\substack{{N(\mathfrak{p})\leq x}\\{d \mid (\pi_p-1)(\bar{\pi}_p-1)}}} 1 \ll \frac{\tau(d_{\rm sp})x}{d}.$$
\item Let $E_{s, t}:~y^2=x^3+sx+t$ be an elliptic curve over $\mathbb{F}_p$ with $st=0$.  We have $\#E_{s, t}(\mathbb{F}_p)=p+1$ or $\#E_{s, t}(\mathbb{F}_p)=(\pi_p -1) (\bar{\pi}_p-1)$ and $N(\pi_p)=p$, where $\pi_p\in \mathbb{Z}[(1+i\sqrt{3})/2]$ or $\mathbb{Z}[i]$.
\item Let $g(d)$ be an arithmetic function satisfying \eqref{bound} with $\beta<1$.
Then we have 
$$\sum_{p \le x} \frac{1}{p}\sum_{\substack{s,t \in \mathbb{F}_{p} \\ st=0}} \sum_{\substack{d|p-1 \\ E_{s,t}(\mathbb{F}_{p})[d] \cong (\mathbb{Z}/d\mathbb{Z})^{2}}} |g(d)|\ll \frac{x}{\log{x}}.$$
\end{enumerate}
\end{lemma}
\begin{proof}
The proofs of (i) and (ii) are identical to the proofs of Propositions 2.2 and 2.3 of \cite{AM}.\\
\noindent (iii) See \cite[Chapter 18, Theorems 4 and 5]{IR}.\\
\noindent (iv) We observe that the condition $E_{s,t}(\mathbb{F}_{p})[d] \cong (\mathbb{Z}/d\mathbb{Z})^{2}$ implies that $d\mid p-1$ and $d^2\mid \#E_{s, t}(\mathbb{F}_p)$. By part (iii) we know  the possibilities for $\#E_{s, t}(\mathbb{F}_p)$. Now if $\#E_{s, t}(\mathbb{F}_p)=p+1$, then we conclude that $d=2$ (since $d\mid p-1$ and $d\mid p+1$). On the other hand if $\#E_{s, t}(\mathbb{F}_p)=(\pi_p-1)(\bar{\pi}_p-1)$ where $\pi_p\in \mathbb{Z}[(1+i\sqrt{3})/2]$ or $\mathbb{Z}[i]$, we let $0<\epsilon<1-\beta$. So by employing (i) and (ii), the sum in (iv) is bounded by 
\begin{equation*}
\sum_{\substack{p\leq x\\p\equiv -1~({\rm mod}~4)}} 1+\sum_{d \le \sqrt{x}+1} |g(d)|  \sum_{\substack{N(\mathfrak{p})\leq x\\ d\mid (\pi_p-1)(\bar{\pi}_p-1)}} 1 \ll \frac{x}{\log{x}}+\frac{x}{\log{x}}\sum_{d \le x^{1/5}} \frac{|g(d)|}{d^{2-\epsilon}}+x\sum_{d>x^{1/5}} \frac{|g(d)|}{d^{2-\epsilon}}\ll \frac{x}{\log{x}}.
\end{equation*}

\end{proof}

We next recall a version of the large sieve inequality for multiplicative characters. 
\begin{lemma}[\bf Gallagher]
\label{largesieve}
Let $M$ and $N$ be positive integers and $(a_n)_{n=M+1}^{M+N}$ be a sequence of complex numbers. Then
$$\sum_{q\leq Q} \frac{q}{\varphi(q)} \sum_{\chi (q)}^{*} \left| \sum_{n=M+1}^{M+N} a_n \chi(n) \right|^2 \ll (N+Q^2) \sum_{n=M+1}^{M+N} |a_n|^2,$$
where $Q$ is any positive real number, and $\sum_{\chi (q)}^{*}$ denotes a sum over all primitive Dirichlet characters $\chi$ modulo $q$.
\end{lemma}
\begin{proof}
See \cite[p. 16]{Gallagher}.
\end{proof}

To state the next lemma, we need to describe some notation. Let
$$\tau_{k, B}(n):=\#\left\{(a_{1},a_{2},\dotsc,a_{k}) \in [1,B]^{k} \cap \mathbb{N}^{k}; ~n=a_{1}a_{2}\dotsm a_{k}\right\}.$$
We also set
\begin{equation*}
\Psi(X,Y):= \sum_{\substack{n \le X \\ p(n) \le Y}} 1,
\end{equation*}
where $p(m)$ is the largest prime factor of $m$.  Note that we define $p(0)=p(\pm 1)=\infty$.

\begin{lemma}[\bf Stephens]
\label{Stephens}
\begin{enumerate}[\upshape (i)]
\item For $k\in \mathbb{N}$, if $B^k\leq x^8$ then $$\sum_{b\leq B^k} \tau_{k, B}(n)^2 < B^k \left(\Psi(B, 9\log{x})\right)^k.$$
\item For a sufficiently large constant $c_1>0$ there exists $c_2>0$ such that if $\exp\left(c_1(\log{x})^{1/2} \right)<B\leq x^8$ then
$$x^{-1/2k} \left( \Psi(B, 9\log{x}) \right)^{1/2}\ll \exp\left(-c_2(\log{x})^{1/2}/\log\log{x}\right),$$
where $$k=\left[ 2\log{x}/\log{B} \right]+1.$$
\item For a sufficiently large constant $c_1>0$ there exists $c_3>0$ such that if $\exp\left(c_1(\log{x})^{1/2} \right)<B\leq x^4$ then
$$x^{-1/k} \left( \Psi(B, 9\log{x}) \right)^{1/2}\ll \exp\left(-c_3(\log{x})^{1/2}/\log\log{x}\right),$$
where $$k=\left[ 4\log{x}/\log{B} \right]+1.$$
\end{enumerate}
\end{lemma}
\begin{proof}
See \cite[Lemmas 8, 9, and 10]{Stephens-II}.
\end{proof}

\begin{lemma}[{\bf Burgess}]
\label{Burgess}
\begin{enumerate}[\upshape (i)]
\item For any prime $p$, non-principle character $\chi$, $r\in \mathbb{N}$, and $B\geq 1$, we have
$$\sum_{b\leq B} \chi(b) \ll B^{1-\frac{1}{r}} p^{\frac{r+1}{4r^2}}\log{p},$$
where the implied constant is absolute.
\item Let $\epsilon>0$, $n>1$, $\chi$ be a non-principal character, $r\in \mathbb{N}$, and $B\geq 1$. Then, if $n$ is cube-free or $r=2$, we have
$$\sum_{b\leq B} \chi(b) \ll B^{1-\frac{1}{r}} n^{\frac{r+1}{4r^2}+\epsilon},$$
where the implied constant may depend on $\epsilon$ and $r$. 
\end{enumerate}
\end{lemma}
\begin{proof}
See \cite[Theorems 1 and 2]{Burgess}.
\end{proof}
\begin{lemma}
\label{FI}
\begin{enumerate}[\upshape (i)]
\item {\bf (Friedlander and Iwaniec)} Let $Q$ and $N$ be positive integers.   Then we have
$$\sum_{\substack{{\chi ({\rm mod}~Q)}}}^{*} \left|\sum_{n\leq N} \chi(n) \right|^4
\ll N^2 Q \log^6{Q},$$
where $*$ denotes a sum over all primitive Dirichlet characters modulo $Q$.
\item Suppose that $Q$ is the product of two distinct primes. Then we have
$$\sum_{\substack{{\chi ({\rm mod}~Q)}\\{\chi \ne \chi_{0}}}} \left|\sum_{n\leq N} \chi(n) \right|^4
\ll N^2 Q \log^6{Q}.$$
\end{enumerate}
\end{lemma}
\begin{proof}
(i) This  is \cite[Lemma 3]{Friedlander&Iwaniec}.  

(ii) Let $Q=pq$ with $p\neq q$. To see that the result is true if the summation is over all non-principal characters, we need to consider the inequality for imprimitive characters.  The only non-principal imprimitive characters modulo $pq$ are of the form $\chi^{\prime}\chi_{0}^{\prime\prime}$ or $\chi_{0}^{\prime}\chi^{\prime\prime}$, where $\chi_{0}^{\prime}$ and $\chi_{0}^{\prime\prime}$ are the principal characters modulo $p$ and $q$, respectively, and $\chi^{\prime}$ and $\chi^{\prime\prime}$ are primitive characters modulo $p$ and $q$, respectively.  Then, partition the summation over all characters into a summation over primitive characters modulo $pq$, primitive characters modulo $p$ and primitive characters modulo $q$.  Hence, the assertion can be obtained by using the triangle inequality and the result for primitive characters in part (i).
\end{proof}

We summarize several elementary estimations that are used in the proofs of next sections.  

\begin{lemma}
\label{phi}
\begin{enumerate}[\upshape (i)]
\item {\emph (Brun-Titchmarsh inequality)} Let $\epsilon>0$. Then for $1\leq d\leq x^{1-\epsilon}$, we have $$\pi(x;d, a)\ll \frac{x}{\varphi(d) \log{x}}.$$
\item Let $\theta <1$ and $\epsilon >0$. Then for $1\leq d \leq x^{1-\epsilon}$, we have
$$\sum_{\substack{p \le x \\ p \equiv 1 \bmod{d}}} \frac{1}{p^\theta} \ll \frac{x^{1-\theta}}{\varphi(d) \log{x}} .$$
\item For $x\geq 3$ and $d\geq 1$ we have $$\sum_{\substack{p \le x \\ p \equiv 1 \bmod{d}}} \frac{1}{p} \ll \frac{\log\log{x}+\log{d}}{\varphi(d)}.$$
\item We have $$\frac{1}{\varphi(d)} \ll \frac{\log\log{d}}{d}.$$
\item Under the assumption of bound \eqref{bound}, for any real $\theta$ we have
$$\sum_{d\leq y} \frac{|g(d)|}{d^\theta}  \ll 1+y^{1+\beta-\theta} (\log{y})^{\gamma+1}.$$
\end{enumerate}
\end{lemma}
\begin{proof}
\begin{enumerate}[\upshape (i)]
\item See \cite[Theorem 7.3.1]{Cojocaru&Murty}.
\item This is a consequence of partial summation and part (i).
\item See \cite[Section 13.1, Exercise 9]{Cojocaru&Murty}.
\item See \cite[p.  267, Theorem 328]{Hardy&Wright}.
\item This comes by straightforward applications of partial summation and bound \eqref{bound}.
\end{enumerate}
\end{proof}


\section{PROOFS OF THEOREM \ref{theorem1} AND COROLLARY \ref{corollaries}}
\subsection{Basic set up} 
Let $\mathcal{C}$ be the family of elliptic curves
\begin{equation*}
E_{a,b}:~y^2=x^3+ax+b,
\end{equation*}
where $|a|\leq A$, $|b|\leq B$, and at least one of $a$ or $b$ is non-zero.  Note that
\begin{equation*}
|\mathcal{C}|=4AB+O(A+B).
\end{equation*}
Let
\begin{equation*}
f(n)=\sum_{d|n} g(d)
\end{equation*}
for all $n \in \mathbb{N}$.  We have
$$\frac{1}{|\mathcal{C}|} \sum_{E_{a, b}\in \mathcal{C}} \sum_{p\leq x} f(i_{E_{a, b}}(p))=\frac{1}{|\mathcal{C}|}\sum_{p\leq x} \sum_{s, t\in \mathbb{F}_p} \frac{|{\rm Aut}_{\mathbb{F}_p}(E_{s, t})|f(i_{E_{s, t}}(p))}{p-1}\sum_{\substack{{|a|\leq A,~ |b|\leq B:~ \exists 1\leq u <p}\\{a\equiv s u^4~({\rm mod}~p)}\\{b\equiv t u^6~({\rm mod}~p)}}} 1.$$
Next 
by applying Remark \ref{automorphism} (i) in the above identity (recall that $p\neq 2, 3$),
we have
$$\frac{1}{|\mathcal{C}|} \sum_{E_{a, b}\in \mathcal{C}} \sum_{p\leq x} f(i_{E_{a, b}}(p))=\frac{2}{|\mathcal{C}|}\sum_{p\leq x} \sum_{s, t\in \mathbb{F}_p^{\times}} \frac{f(i_{E_{s, t}}(p))}{p-1}\sum_{\substack{{|a|\leq A,~ |b|\leq B,~ \exists 1\leq u <p}\\{a\equiv s u^4~({\rm mod}~p)}\\{b\equiv t u^6~({\rm mod}~p)}}} 1+{\rm Error ~Term} ~1,$$
where 
\begin{equation}
\label{errorterm1}
{\rm Error~Term~1}= \frac{1}{|\mathcal{C}|}\sum_{p \le x}\sum_{\substack{s,t \in \mathbb{F}_{p} \\ st=0}} \frac{|{\rm Aut}_{\mathbb{F}_p}(E_{s, t})|f(i_{E_{s,t}}(p))}{p-1}\sum_{\substack{|a| \le A,|b| \le B \\ ab \equiv 0 \bmod{p}}} 1.
\end{equation}
Now by considering 
$$\sum_{\substack{{|a|\leq A,~ |b|\leq B,~ \exists 1\leq u <p}\\{a\equiv s u^4~({\rm mod}~p)}\\{b\equiv t u^6~({\rm mod}~p)}}} 1=\frac{2AB}{p}+ \left( \sum_{\substack{{|a|\leq A,~ |b|\leq B,~ \exists 1\leq u <p}\\{a\equiv s u^4~({\rm mod}~p)}\\{b\equiv t u^6~({\rm mod}~p)}}} 1 -\frac{2AB}{p}\right)$$
and applying it in the previous identity we arrive at
$$\frac{1}{|\mathcal{C}|} \sum_{E_{a, b}\in \mathcal{C}} \sum_{p\leq x} f(i_{E_{a, b}}(p))={\rm The~Main~Term}+{\rm Error~Term~1}+{\rm Error~Term~2},$$
where
$${\rm The~Main~Term}=\frac{4AB}{|\mathcal{C}|}\sum_{p\leq x} \sum_{s, t\in \mathbb{F}_p^\times} \frac{f(i_{E_{s, t}}(p))}{p(p-1)}$$
and
\begin{eqnarray*}
{\rm Error~Term~2}&=&  \frac{2}{|\mathcal{C}|}\sum_{p\leq x} \sum_{s, t\in \mathbb{F}_p^{\times}} \frac{f(i_{E_{s, t}}(p))}{p-1}\left( \sum_{\substack{{|a|\leq A,~ |b|\leq B,~ \exists 1\leq u <p}\\{a\equiv s u^4~({\rm mod}~p)}\\{b\equiv t u^6~({\rm mod}~p)}}} 1-\frac{2AB}{p}\right).
\end{eqnarray*}

\subsection{The Main Term}
\label{themainterm}
We have 
\begin{eqnarray*}
{\rm The~Main~Term}=\frac{4AB}{|\mathcal{C}|}\sum_{p\leq x} \sum_{s, t\in \mathbb{F}_p^\times} \frac{f(i_{E_{s, t}}(p))}{p(p-1)}&=&\frac{4AB}{|\mathcal{C}|}\sum_{p\leq x}\frac{1}{p(p-1)} \sum_{s, t\in \mathbb{F}_p^\times} \sum_{d\mid i_{E_{s, t}}(p)} g(d)\\
&=&\frac{4AB}{|\mathcal{C}|}\sum_{p\leq x}\frac{1}{p(p-1)}  \sum_{d\mid p-1} g(d)\#\tilde{\mathcal{S}}_d(p).\\
\end{eqnarray*}

Let
\begin{equation*}
G_1(p)=\sum_{\substack{d\mid p-1 \\ d \le \sqrt{p}+1}} \frac{g(d)}{d\psi(d) \varphi(d)} \qquad {\textrm{and}} \qquad G_2(p)=\sum_{\substack{d\mid p-1 \\ d \le \sqrt{p}+1}} |g(d)|.
\end{equation*}
By using these notations and employing Lemma \ref{lemma:1} we have
\begin{align*}
{\rm The~Main~Term} &=\frac{4AB}{|\mathcal{C}|}\left( \sum_{p\leq x} G_1(p) +O\left( \sum_{p\leq x} \frac{G_2(p)}{\sqrt{p}} \right)\right)\\
&=\frac{4AB}{|\mathcal{C}|}\Big(\mathscr{S}_{1}+O(\mathscr{S}_{2})\Big).
\end{align*}
\subsubsection{Estimation of $\mathscr{S}_1$}
\label{Sigma1} 
Let $\alpha \in \mathbb{R}_{>0}$ be fixed.  The Siegel-Walfisz Theorem implies
\begin{equation*}
\pi(x;d,1)=\frac{\li(x)}{\varphi(d)}+O\left(\frac{x}{(\log x)^{C}}\right)
\end{equation*}
for any $d \le (\log x)^{\alpha}$ and any $C > 0$.  
Then, by the Brun-Titchmarsh inequality (Lemma \ref{phi} (i)),  the fact that $\psi(d) \ge d$, and \eqref{bound},
we have
\begin{align*}
\mathscr{S}_{1} &= \sum_{d \le (\log x)^{\alpha}} \frac{g(d)\pi(x;d,1)}{d\psi(d)\varphi(d)}+\sum_{(\log x)^{\alpha} < d \le \sqrt{x}+1} \frac{g(d)\pi(x;d,1)}{d\psi(d)\varphi(d)} \\
&=\li(x) \sum_{d \ge 1} \frac{g(d)}{d\psi(d)\varphi(d)^2}+O\left(\frac{x}{(\log x)^{C}}\sum_{d \ge 1} \frac{|g(d)|}{d\psi(d)\varphi(d)}\right)+O\left(\frac{x}{\log x}\sum_{d>(\log x)^{\alpha}} \frac{|g(d)|}{d\psi(d)\varphi(d)^{2}}\right).
\end{align*}
Note that, for any $\varepsilon > 0$, we have
\begin{equation*}
\sum_{d > y} \frac{|g(d)|}{d\psi(d)\varphi(d)} \ll \sum_{d > y} \frac{|g(d)|}{d^{3-\frac{\varepsilon}{2}}} \ll \frac{1}{y^{2-\beta-\varepsilon}}.
\end{equation*}
Thus, for $\beta < 2$, 
\begin{equation*}
c_{0}(f):=\sum_{d \ge 1} \frac{g(d)}{d\psi(d)\varphi(d)^2}
\end{equation*}
is a constant and
\begin{equation*}
\mathscr{S}_{1}=c_{0}(f)\li(x)+O\left(\frac{x}{(\log x)^{C^{\prime}}}\right), 
\end{equation*}
where $C^{\prime}:=C^\prime(C, \alpha, \beta, \varepsilon)$ is an appropriate positive constant. Since $\alpha$ is arbitrary, we can choose $\alpha$ so that $C^\prime$ is any constant bigger than $1$. So
\begin{equation}\label{eq:3.2}
\mathscr{S}_{1}=c_{0}(f)\li(x)+O\left(\frac{x}{(\log x)^{c}}\right), 
\end{equation}
where $c$ can be chosen as any number bigger than $1$.
\subsubsection{Estimation of $\mathscr{S}_{2}$}
We first employ the Brun-Titchmarsh inequality (Lemma \ref{phi} (i)) and \eqref{bound} to deduce
\begin{equation}\label{eq:3.3}
\sum_{p \le x} G_{2}(p) = \sum_{d \le \sqrt{x}+1} |g(d)|\pi(x;d,1) \\
\ll \begin{cases}
x^{1+\frac{\beta}{2}}(\log x)^{\gamma-1}\log \log x \qquad &\text{if } \beta \ne 0 \\
x^{1+\frac{\beta}{2}}(\log x)^{\gamma}\log \log x &\text{if } \beta = 0
\end{cases}\\
\end{equation}
By partial summation and (\ref{eq:3.3}), we have
\begin{align}\label{eq:3.4}
\mathscr{S}_{2}&=\sum_{p \le x} \frac{G_{2}(p)}{\sqrt{p}}
\ll x^{\frac{1+\beta}{2}}(\log x)^{\gamma}\log \log x.
\end{align}

In conclusion, since $\beta < 1$
\begin{equation}
\label{mterm}
{\rm The~Main~ Term}=\frac{4AB}{|\mathcal{C}|}\left(c_{0}(f)\li(x)
+O\left(\frac{x}{(\log x)^{c}}\right)\right),
\end{equation}
where $c$ can be taken as any number bigger than $1$.

\subsection{Error Term 1}
Recall the expression \eqref{errorterm1} for Error Term 1. We have
\begin{align*}
{\rm Error ~Term ~1}&\ll \frac{1}{|\mathcal{C}|}\sum_{p \le x}\sum_{\substack{s,t \in \mathbb{F}_{p} \\ st=0}} \frac{|f(i_{E_{s,t}}(p))|}{p}\left(\frac{AB}{p}+A+B\right) \\
&\ll \sum_{p \le x} \frac{1}{p^{2}}\sum_{\substack{s,t \in \mathbb{F}_{p} \\ st=0 }} \sum_{\substack{d|p-1  \\ E_{s,t}(\mathbb{F}_{p})[d] \cong (\mathbb{Z}/d\mathbb{Z})^{2}}}|g(d)|+\left(\frac{1}{A}+\frac{1}{B}\right)\sum_{p \le x} \frac{1}{p}   \sum_{\substack{s,t \in \mathbb{F}_{p} \\ st=0 }} \sum_{\substack{d|p-1  \\ E_{s,t}(\mathbb{F}_{p})[d] \cong (\mathbb{Z}/d\mathbb{Z})^{2}}}|g(d)|.
\end{align*}
An application of part (iv) of Lemma \ref{claim}  in the latter sum yields
\begin{align}
\label{term1}
{\rm Error~Term~1}&\ll \sum_{p \le x} \frac{1}{p}\sum_{\substack{d|p-1 \\ d \le \sqrt{p}+1}} |g(d)|
+\left(\frac{1}{A}+\frac{1}{B}\right)\frac{x}{\log{x}}.
\end{align}
By employing Lemma \ref{phi} (iii) and (iv) and usual estimates, the first of these summations is bounded as follows.
\begin{equation}
\label{firstsum}
\sum_{p \le x} \frac{1}{p}\sum_{\substack{d|p-1 \\ d \le \sqrt{p}+1}} |g(d)|=\sum_{d \le \sqrt{x}+1} |g(d)| \sum_{\substack{p \le x \\ p \equiv 1 \bmod{d}}} \frac{1}{p} \\
\ll(\log \log x)(\log x)\sum_{d \le \sqrt{x}+1} \frac{|g(d)|}{d}.
\end{equation}
From applying part (v) of Lemma  \ref{phi} in  \eqref{firstsum}
we have
\begin{equation}
\label{Error1}
{\rm Error ~Term ~1} \ll x^{\frac{\beta}{2}}(\log x)^{\gamma+2}(\log \log x)+\left(\frac{1}{A}+\frac{1}{B}\right)\frac{x}{\log{x}}.
\end{equation}

\subsection{Error Term 2}


\indent We summarize the main result of this section in the following lemma, which can be considered as a generalization and an improvement of Lemma 6 of \cite{BCD}.
\begin{lemma}\label{lemma:2}
Let $r\in \mathbb{N}$, $0 \le \beta<3/2$, $\gamma\in \mathbb{R}_{\ge 0}$, and  $g:\mathbb{N} \to \mathbb{C}$ be a function such that
\begin{equation*}
\sum_{d \le x} |g(d)| \ll x^{1+\beta}(\log x)^{\gamma}.
\end{equation*}
Then there are positive constants $c_1$ and $c_2$ such that if  $A,B > \exp(c_{1}(\log x)^{1/2})$ 
we have
\begin{align*}
&\frac{2}{|\mathcal{C}|}\sum_{p \le x} \sum_{d|p-1} g(d) \sum_{\substack{1 \le s,t < p \\ E_{s,t}(\mathbb{F}_{p})[d] \cong (\mathbb{Z}/d\mathbb{Z})^{2}}} \frac{1}{p-1}\left(\sum_{\substack{|a| \le A, |b| \le B:\\ \exists 1 \le u < p \\ a \equiv su^{4} \bmod{p} \\b \equiv tu^{6} \bmod{p}}}1-\frac{2AB}{p}\right) \\
&\qquad \qquad \ll x^{\frac{\beta-1}{2}}(\log{x})^{\gamma+1}\log\log{x}+(\log{x})^\gamma \log\log{x}+\left(\frac{1}{A}+\frac{1}{B}\right)\left(\frac{x}{\log{x}}+x^{\frac{1+\beta}{2}}(\log x)^{\gamma}\log\log{x}\right)\\
&\qquad \qquad \qquad+x\exp\left(-c_{2}\frac{(\log x)^{1/2}}{\log \log x}\right)+\left(\frac{1}{A^{1/r}}+\frac{1}{B^{1/r}}\right)x^{\frac{1+\beta}{2}+\frac{r+1}{4r^{2}}}(\log x)^{\gamma+1}\log \log x\\
&\qquad \qquad \qquad+\frac{1}{\sqrt{AB}}\left(x^{\frac{3}{2}}(\log x)^{2}+  x^{1+\frac{\beta}{2}}(\log x)^{\gamma+3}(\log \log x)^{\frac{5}{4}}   +x^{\frac{5+2\beta}{4}}(\log x)^{\gamma+3} \log\log{x}    \right).
\end{align*}
\end{lemma}
\begin{proof}
Throughout, $\chi$, with or without subscript, will denote a character modulo $p$.  As usual, $\chi_{0}$ will be the principal character modulo $p$.  Let $p$ be a fixed prime, and let $s,t \in \mathbb{F}_{p}^{\times}$ be fixed.  By \cite[Equation (12)]{BCD}, we have
\begin{equation*}
\sum_{\substack{|a| \le A,|b| \le B: \\ \exists 1 \le u < p \\ a \equiv su^{4} \bmod{p} \\ b \equiv tu^{6} \bmod{p}}} 1 = \frac{1}{2(p-1)}\sum_{\substack{\chi_{1},\chi_{2} \\ \chi_{1}^{4}\chi_{2}^{6}=\chi_{0}}} \chi_{1}(s)\chi_{2}(t)\mathcal{A}(\overline{\chi_{1}})\mathcal{B}(\overline{\chi_{2}}),
\end{equation*}
where
\begin{equation*}
\mathcal{A}(\chi):=\sum_{|a| \le A} \chi(a) \qquad \qquad 
~~~~{\rm and}~~~~\qquad
\mathcal{B}(\chi):=\sum_{|b| \le B} \chi(b).
\end{equation*}
We use the identity
\begin{align*}
&\frac{1}{2(p-1)}\sum_{\substack{\chi_{1},\chi_{2} \\ \chi_{1}^{4}\chi_{2}^{6}=\chi_{0}}} \chi_{1}(s)\chi_{2}(t)\mathcal{A}(\overline{\chi_{1}})\mathcal{B}(\overline{\chi_{2}}) \\
&\qquad = \frac{1}{2(p-1)}\chi_{0}(s)\chi_{0}(t)\mathcal{A}(\overline{\chi_{0}})\mathcal{B}(\overline{\chi_{0}})+\frac{1}{2(p-1)}\sum_{\substack{\chi_{0} \ne \chi_{2} \\ \chi_{2}^{6}=\chi_{0}}} \chi_{0}(s)\chi_{2}(t)\mathcal{A}(\overline{\chi_{0}})\mathcal{B}(\overline{\chi_{2}})\\
&\qquad \qquad+\frac{1}{2(p-1)}\sum_{\substack{\chi_{1}\ne \chi_{0} \\ \chi_{1}^{4}=\chi_{0}}} \chi_{1}(s)\chi_{0}(t)\mathcal{A}(\overline{\chi_{1}})\mathcal{B}(\overline{\chi_{0}})+\frac{1}{2(p-1)}\sum_{\substack{\chi_{1} \ne \chi_{0} \\ \chi_{2} \ne \chi_{0} \\ \chi_{1}^{4}\chi_{2}^{6}=\chi_{0}}} \chi_{1}(s)\chi_{2}(t)\mathcal{A}(\overline{\chi_{1}})\mathcal{B}(\overline{\chi_{2}})
\end{align*}
and note that 
\begin{equation*}
\frac{1}{2(p-1)}\chi_{0}(s)\chi_{0}(t)\mathcal{A}(\overline{\chi_{0}})\mathcal{B}(\overline{\chi_{0}}) =\frac{1}{2(p-1)}\sum_{|a| \le A} \chi_{0}(a)\sum_{|b| \le B} \chi_{0}(b)
= \frac{2AB}{p}+O\left(\frac{AB}{p^{2}}+\frac{A+B}{p}\right).
\end{equation*}
Therefore,
\begin{align*}
&\frac{2}{|\mathcal{C}|} \sum_{p \le x}\sum_{d|p-1} g(d) \sum_{\substack{1 \le s,t < p \\ E_{s,t}(\mathbb{F}_{p})[d] \cong (\mathbb{Z}/d\mathbb{Z})^{2}}} \frac{1}{p-1}\left(\sum_{\substack{|a| \le A, |b| \le B:\\ \exists 1 \le u < p \\ a \equiv su^{4} \bmod{p} \\b \equiv tu^{6} \bmod{p}}}1-\frac{2AB}{p}\right) \\
\qquad &= \frac{2}{|\mathcal{C}|} \sum_{p \le x}\sum_{d|p-1} g(d) \sum_{\substack{1 \le s,t < p \\ E_{s,t}(\mathbb{F}_{p})[d] \cong (\mathbb{Z}/d\mathbb{Z})^{2}}} \frac{1}{p-1} \Bigg(O\left(\frac{AB}{p^{2}}+\frac{A+B}{p}\right)+\frac{1}{2(p-1)}\sum_{\substack{ \chi_{2} \ne \chi_{0} \\ \chi_{2}^{6}=\chi_{0}}} \chi_{0}(s)\chi_{2}(t)\mathcal{A}(\overline{\chi_{0}})\mathcal{B}(\overline{\chi_{2}})\\
&\qquad \qquad+\frac{1}{2(p-1)}\sum_{\substack{\chi_{1}\ne \chi_{0} \\ \chi_{1}^{4}=\chi_{0}}} \chi_{1}(s)\chi_{0}(t)\mathcal{A}(\overline{\chi_{1}})\mathcal{B}(\overline{\chi_{0}})+\frac{1}{2(p-1)}\sum_{\substack{\chi_{1} \ne \chi_{0} \\ \chi_{2} \ne \chi_{0} \\ \chi_{1}^{4}\chi_{2}^{6}=\chi_{0}}} \chi_{1}(s)\chi_{2}(t)\mathcal{A}(\overline{\chi_{1}})\mathcal{B}(\overline{\chi_{2}})\Bigg)\\
&=:\Sigma_{1}+\Sigma_{2}+\Sigma_{3}+\Sigma_{4}.
\end{align*}
We will evaluate each summation separately.
\subsubsection{Estimation of $\Sigma_1$}\label{sigma1}
\indent We have
\begin{align*}
\Sigma_{1}&:=\frac{2}{|\mathcal{C}|} \sum_{p \le x}\sum_{d|p-1} g(d) \sum_{\substack{1 \le s,t < p \\ E_{s,t}(\mathbb{F}_{p})[d] \cong (\mathbb{Z}/d\mathbb{Z})^{2}}} \frac{1}{p-1} O\left(\frac{AB}{p^{2}}+\frac{A+B}{p}\right) \\
&\ll \frac{1}{|\mathcal{C}|} \sum_{p \le x}\left(\frac{AB}{p^{3}}+\frac{A+B}{p^{2}}\right)\sum_{d|p-1} |g(d)| \sum_{\substack{1 \le s,t < p \\ E_{s,t}(\mathbb{F}_{p})[d] \cong (\mathbb{Z}/d\mathbb{Z})^{2}}} 1\\
&\ll \frac{AB}{|\mathcal{C}|} \sum_{p \le x}\frac{1}{p^{3}}\sum_{\substack{d|p-1 \\ d \le \sqrt{p}+1}} |g(d)|\left(\frac{p(p-1)}{d\psi(d)\varphi(d)}+O(p^{3/2})\right)+\left(\frac{A+B}{|\mathcal{C}|}\right) \sum_{p \le x}\frac{1}{p^{2}}\sum_{\substack{d|p-1 \\ d \le \sqrt{p}+1}} |g(d)|\left(\frac{p(p-1)}{d\psi(d)\varphi(d)}+O(p^{3/2})\right).
\end{align*}
We denote the first summation by $\Sigma_{1,1}$ and the second by $\Sigma_{1,2}$.  By partial summation and \eqref{eq:3.3}, we have
\begin{align}\label{eq:4.1}
\Sigma_{1,1} &\ll 
x^{\frac{\beta-1}{2}}(\log{x})^{\gamma+1}\log\log{x}+(\log{x})^\gamma \log\log{x}
\end{align}
as $\beta < 3/2$.\\
\indent By Equations (\ref{eq:3.2}) and (\ref{eq:3.4}), we have
\begin{align}\label{eq:4.2}
\Sigma_{1,2} &\ll \left(\frac{1}{A}+\frac{1}{B}\right)\left(\sum_{p \le x} \sum_{\substack{d|p-1 \\ d \le \sqrt{p}+1}} \frac{|g(d)|}{d\psi(d)\varphi(d)}+\sum_{p \le x} \frac{1}{p^{1/2}}\sum_{\substack{d|p-1 \\ d \le \sqrt{p}+1}}|g(d)|\right) \notag\\
&\ll \left(\frac{1}{A}+\frac{1}{B}\right)\left(\frac{x}{\log{x}}+x^{\frac{1+\beta}{2}}(\log x)^{\gamma}\log\log{x}\right).
\end{align}
Therefore, $\Sigma_{1}$ is bounded by the error terms in the lemma.
\subsubsection{Estimations of $\Sigma_2$ and $\Sigma_3$}\label{subsubsec:3.4.2}
\indent For $\Sigma_{2}$, we have
\begin{align*}
\Sigma_{2}&:= \frac{1}{|\mathcal{C}|} \sum_{p \le x}\sum_{\substack{d|p-1 \\ d \le \sqrt{p}+1}} g(d) \sum_{\substack{1 \le s,t < p \\ E_{s,t}(\mathbb{F}_{p})[d] \cong (\mathbb{Z}/d\mathbb{Z})^{2}}} \frac{1}{(p-1)^{2}}\sum_{\substack{\chi_{2} \ne \chi_{0} \\ \chi_{2}^{6}=\chi_{0}}} \chi_{0}(s)\chi_{2}(t)\mathcal{A}(\overline{\chi_{0}})\mathcal{B}(\overline{\chi_{2}})\\
&\ll\frac{1}{|\mathcal{C}|}\sum_{p \le x}\sum_{\substack{d|p-1 \\ d \le \sqrt{p}+1}} |g(d)| \sum_{\substack{1 \le s,t < p \\ E_{s,t}(\mathbb{F}_{p})[d] \cong (\mathbb{Z}/d\mathbb{Z})^{2}}} \frac{1}{p^{2}}\sum_{\substack{\chi_{2} \ne \chi_{0} \\ \chi_{2}^{6}=\chi_{0}}} |\mathcal{B}(\overline{\chi_{2}})|\sum_{\substack{-A \le a \le A \\ p \nmid a}}1\\
&\ll \frac{A}{|\mathcal{C}|}\sum_{p \le x} \frac{1}{p^{2}}\sum_{\substack{d|p-1 \\ d \le \sqrt{p}+1}} |g(d)| \sum_{\substack{\chi_{2} \ne \chi_{0} \\ \chi_{2}^{6}=\chi_{0}}} |\mathcal{B}(\overline{\chi_{2}})|\sum_{\substack{1 \le s,t < p \\ E_{s,t}(\mathbb{F}_{p})[d] \cong (\mathbb{Z}/d\mathbb{Z})^{2}}}1.
\end{align*}
By Lemma \ref{lemma:1}, we have
\begin{align*}
\Sigma_{2} &\ll \frac{1}{B}\sum_{p \le x} \frac{1}{p^{2}}\sum_{\substack{d|p-1 \\ d \le \sqrt{p}+1}} |g(d)| \sum_{\substack{\chi_{2} \ne \chi_{0} \\ \chi_{2}^{6}=\chi_{0}}} |\mathcal{B}(\overline{\chi_{2}})|\left(\frac{p(p-1)}{d\psi(d)\varphi(d)}+O\left(p^{3/2}\right)\right) \\
&\ll \frac{1}{B}\sum_{p \le x} \sum_{\substack{d|p-1 \\ d \le \sqrt{p}+1}} \frac{|g(d)|}{d\psi(d)\varphi(d)} \sum_{\substack{\chi_{2} \ne \chi_{0} \\ \chi_{2}^{6}=\chi_{0}}} |\mathcal{B}(\overline{\chi_{2}})|+\frac{1}{B}\sum_{p \le x} \frac{1}{p^{1/2}}\sum_{\substack{d|p-1 \\ d \le \sqrt{p}+1}} |g(d)| \sum_{\substack{\chi_{2} \ne \chi_{0} \\ \chi_{2}^{6}=\chi_{0}}} |\mathcal{B}(\overline{\chi_{2}})|\\
&=:\Sigma_{2,1}+\Sigma_{2.2}.
\end{align*}
Now,
\begin{equation}\label{eq:4.3}
\Sigma_{2,1} = \frac{1}{B}\sum_{d \le \sqrt{x}+1} \frac{|g(d)|}{d\psi(d)\varphi(d)} \sum_{\substack{p \le x \\ p \equiv 1 \bmod{d}}} \sum_{\substack{\chi_{2} \ne \chi_{0} \\ \chi_{2}^{6}=\chi_{0}}} |\mathcal{B}(\overline{\chi_{2}})|.
\end{equation}
Let $k=[2\log x/\log B]+1$.  By H\"{o}lder's inequality, we have
\begin{align}\label{eq:4.4}
\sum_{\substack{p \le x \\ p \equiv 1 \bmod{d}}} \sum_{\substack{\chi_{2} \ne \chi_{0} \\ \chi_{2}^{6}=\chi_{0}}} |\mathcal{B}(\overline{\chi_{2}})| &\le \left(\sum_{\substack{p \le x \\ p \equiv 1 \bmod{d}}} \sum_{\substack{\chi_{2} \ne \chi_{0} \\ \chi_{2}^{6}=\chi_{0}}}1\right)^{1-\frac{1}{2k}}\left(\sum_{\substack{p \le x \\ p \equiv 1 \bmod{d}}}\sum_{\substack{\chi_{2} \ne \chi_{0} \\ \chi_{2}^{6}=\chi_{0}}} \left|\sum_{b \le B} \chi_{2}(b)\right|^{2k}\right)^{\frac{1}{2k}} \notag \\
&\ll (\pi(x;d,1))^{1-\frac{1}{2k}}\left(\sum_{p \le x}\sum_{\substack{\chi _{2}\ne \chi_{0}}} \left|\sum_{b \le B^{k}} \tau_{k, B}(b)\chi_{2}(b)\right|^{2}\right)^{\frac{1}{2k}},
\end{align}
where $\tau_{k, B}(n):=\#\left\{(a_{1},a_{2},\dotsc,a_{k}) \in [1,B]^{k} \cap \mathbb{N}^{k}:n=a_{1}a_{2}\dotsm a_{k}\right\}$.  By Lemma \ref{largesieve}, we have
\begin{equation}\label{eq:4.5}
\sum_{p \le x}\sum_{\substack{\chi \ne \chi_{0}}} \left|\sum_{b \le B^{k}} \tau_{k, B}(b)\chi(b)\right|^{2} \ll (x^{2}+B^{k})\sum_{b \le B^{k}} \tau_{k, B}(b)^{2}.
\end{equation}
Suppose $k=1$.  That is, $B > x^{2}$.  Then, we obtain
\begin{equation*}
\sum_{p \le x}\sum_{\substack{\chi_2 \ne \chi_{0}}} \left|\sum_{b \le B^{k}} \tau_{1}^{B}(b)\chi_{2}(b)\right|^{2} \ll B^{2}.
\end{equation*}
Therefore from \eqref{eq:4.4} we have
\begin{equation*}
\sum_{\substack{p \le x \\ p \equiv 1 \bmod{d}}} \sum_{\substack{\chi_{2} \ne \chi_{0} \\ \chi_{2}^{6}=\chi_{0}}} |\mathcal{B}(\overline{\chi_{2}})| \ll B\frac{x^{1/2}}{\varphi(d)^{1/2}(\log x)^{1/2}}
\end{equation*}
after using Lemma \ref{phi} (i).  Substituting this into (\ref{eq:4.3}), we obtain
\begin{equation*}
\Sigma_{2,1} \ll \frac{x^{1/2}}{(\log x)^{1/2}}\sum_{d \le x} \frac{|g(d)|}{d\psi(d)\varphi(d)^{3/2}} \ll \frac{x^{1/2}}{(\log x)^{1/2}},
\end{equation*}
as $\beta < 3/2$ and the summation above was previously determined to be a constant.\\
\indent Now suppose $k=[2\log{x}/\log{B}]+1 > 1$.  Then $B\leq x^2$ and $x^2 < B^k \leq Bx^2 \leq x^4$. 
Then, by Lemma \ref{Stephens} (i) and (ii), (\ref{eq:4.4}), (\ref{eq:4.5}), and the trivial bound for $\pi(x;d,1)$, we have
\begin{align}\label{eq:4.6}
\sum_{\substack{p \le x \\ p \equiv 1 \bmod{d}}} \sum_{\substack{\chi_{2} \ne \chi_{0} \\ \chi_{2}^{6}=\chi_{0}}} |\mathcal{B}(\overline{\chi_{2}})| &\ll \left(\frac{x}{d}\right)^{1-\frac{1}{2k}}\left((x^{2}+B^{k})B^{k}\Big(\Psi(B,9\log x)\Big)^{k}\right)^{\frac{1}{2k}} \notag\\
&\ll B\frac{x}{d^{3/4}} x^{-\frac{1}{2k}}\Big(\Psi(B,9\log x)\Big)^{1/2} \notag\\
&\ll B\frac{x}{d^{3/4}}\exp\left(-c_{2}\frac{(\log x)^{1/2}}{\log \log x}\right),
\end{align}
where $c_{2}>0$ if $c_1$ is sufficiently large.
Substituting (\ref{eq:4.6}) into (\ref{eq:4.3}), we obtain
\begin{equation*}
\Sigma_{2,1} \ll x\exp\left(-c_{2}\frac{(\log x)^{1/2}}{\log \log x}\right)\sum_{d \le x} \frac{|g(d)|}{d^{7/4}\psi(d)\varphi(d)} \ll x\exp\left(-c_{2}\frac{(\log x)^{1/2}}{\log \log x}\right),
\end{equation*}
as $\beta < 3/2$.\\
\indent For $\Sigma_{2,2}$, 
by Lemma \ref{Burgess} (i),
\eqref{bound},
and Lemma \ref{phi} (i), (ii), and (v), we have
\begin{align*}
\Sigma_{2,2} &= \frac{1}{B}\sum_{p \le x} \frac{1}{p^{1/2}}\sum_{\substack{d|p-1 \\ d \le \sqrt{p}+1}} |g(d)| \sum_{\substack{\chi_{2} \ne \chi_{0} \\ \chi_{2}^{6}=\chi_{0}}} |\mathcal{B}(\overline{\chi_{2}})(b)| \ll \frac{1}{B}\sum_{d \le \sqrt{x}+1} |g(d)| \sum_{\substack{p \le x \\ p \equiv 1 \bmod{d}}} \frac{1}{p^{1/2}} \sum_{\substack{\chi_{2} \ne \chi_{0} \\ \chi_{2}^{6}=\chi_{0}}} \left|\sum_{b \le B} \chi_{2}(b)\right| \\
&\ll \frac{1}{B^{\frac{1}{r}}} \sum_{d \le \sqrt{x}+1} |g(d)| \sum_{\substack{p \le x \\ p \equiv 1 \bmod{d}}} p^{\frac{-2r^{2}+r+1}{4r^{2}}}\log p \sum_{\substack{\chi_{2} \ne \chi_{0} \\ \chi_{2}^{6}=\chi_{0}}} 1 \ll \frac{x^{\frac{1}{2}+\frac{r+1}{4r^{2}}}\log\log{x}}{B^{\frac{1}{r}}}\sum_{d \le \sqrt{x}+1} \frac{|g(d)|}{d} \\
&\ll \frac{x^{\frac{1+\beta}{2}+\frac{r+1}{4r^{2}}}(\log x)^{\gamma+1}\log \log x}{B^{\frac{1}{r}}}.
\end{align*}
\indent The proof of the bound for $\Sigma_{2}$ gives us the same bound for $\Sigma_{3}$, \textit{mutatis mutandis}.
\subsubsection{Estimation of $\Sigma_4$}
\indent For $\Sigma_{4}$, we have
\begin{align*}
\Sigma_{4} &= \frac{2}{|\mathcal{C}|} \sum_{p \le x}\sum_{\substack{d|p-1 \\ d \le \sqrt{p}+1}} g(d) \sum_{\substack{1 \le s,t < p \\ E_{s,t}(\mathbb{F}_{p})[d] \cong (\mathbb{Z}/d\mathbb{Z})^{2}}} \frac{1}{2(p-1)^{2}}\sum_{\substack{\chi_{1} \ne \chi_{0} \\ \chi_{2} \ne \chi_{0} \\ \chi_{1}^{4}\chi_{2}^{6}=\chi_{0}}} \chi_{1}(s)\chi_{2}(t)\mathcal{A}(\overline{\chi_{1}})\mathcal{B}(\overline{\chi_{2}})\\
&= \frac{1}{|\mathcal{C}|} \sum_{d \le \sqrt{x}+1} g(d) \sum_{\substack{p \le x \\ p \equiv 1 \bmod{d}}} \frac{1}{(p-1)^{2}} \sum_{\substack{\chi_{1} \ne \chi_{0} \\ \chi_{2} \ne \chi_{0} \\ \chi_{1}^{4}\chi_{2}^{6}=\chi_{0}}} \mathcal{A}(\overline{\chi_{1}})\mathcal{B}(\overline{\chi_{2}}) \sum_{\substack{1 \le s,t < p \\ E_{s,t}(\mathbb{F}_{p})[d] \cong (\mathbb{Z}/d\mathbb{Z})^{2}}} \chi_{1}(s)\chi_{2}(t)\\
&= \frac{1}{|\mathcal{C}|} \sum_{d \le \sqrt{x}+1} g(d) \sum_{\substack{p \le x \\ p \equiv 1 \bmod{d}}} \frac{1}{(p-1)^{2}} \sum_{\substack{\chi_{1} \ne \chi_{0} \\ \chi_{2} \ne \chi_{0} \\ \chi_{1}^{4}\chi_{2}^{6}=\chi_{0}}} \mathcal{A}(\overline{\chi_{1}})\mathcal{B}(\overline{\chi_{2}})\mathcal{W}_{p,d}(\chi_{1},\chi_{2}),
\end{align*}
where
\begin{equation*}
\mathcal{W}_{p,d}(\chi_{1},\chi_{2}):=\sum_{\substack{1 \le s,t < p \\ E_{s,t}(\mathbb{F}_{p})[d] \cong (\mathbb{Z}/d\mathbb{Z})^{2}}} \chi_{1}(s)\chi_{2}(t).
\end{equation*}
Applying the Cauchy-Schwarz inequality twice, we obtain
\begin{equation*}
\left|\sum_{\substack{\chi_{1} \ne \chi_{0} \\ \chi_{2} \ne \chi_{0} \\ \chi_{1}^{4}\chi_{2}^{6}=\chi_{0}}} \mathcal{A}(\overline{\chi_{1}})\mathcal{B}(\overline{\chi_{2}})\mathcal{W}_{p,d}(\chi_{1},\chi_{2})\right|^{4} \le \left(\sum_{\substack{\chi_{1} \ne \chi_{0} \\ \chi_{2} \ne \chi_{0} \\ \chi_{1}^{4}\chi_{2}^{6}=\chi_{0}}} \bigl|\mathcal{W}_{p,d}(\chi_{1},\chi_{2})\bigr|^{2}\right)^{2}\left(\sum_{\substack{\chi_{1} \ne \chi_{0} \\ \chi_{2} \ne \chi_{0} \\ \chi_{1}^{4}\chi_{2}^{6}=\chi_{0}}} |\mathcal{A}(\chi_{1})|^{4}\right)\left(\sum_{\substack{\chi_{1} \ne \chi_{0} \\ \chi_{2} \ne \chi_{0} \\ \chi_{1}^{4}\chi_{2}^{6}=\chi_{0}}} |\mathcal{B}(\chi_{2})|^{4}\right).
\end{equation*}
By Lemma \ref{FI}, we have
\begin{equation*}
\sum_{\chi_{1} \ne \chi_{0}} \left|\sum_{a \le A} \chi_{1}(a)\right|^{4} \ll A^{2}p(\log p)^{6}.
\end{equation*}
Hence,
\begin{align*}
\sum_{\substack{\chi_{1} \ne \chi_{0} \\ \chi_{2} \ne \chi_{0} \\ \chi_{1}^{4}\chi_{2}^{6}=\chi_{0}}} |\mathcal{A}(\chi_{1})|^{4} &= \sum_{\substack{\chi_{1} \ne \chi_{0} \\ \chi_{2} \ne \chi_{0} \\ \chi_{1}^{4}\chi_{2}^{6}=\chi_{0}}} \left|\sum_{|a| \le A} \chi_{1}(a)\right|^{4} \le 16\sum_{\chi_{1} \ne \chi_{0}} \left|\sum_{a \le A} \chi_{1}(a)\right|^{4}\sum_{\substack{\chi_{2} \ne \chi_{0} \\ \chi_{1}^{4}\chi_{2}^{6}=\chi_{0}}}1\\
&\ll \sum_{\chi_{1} \ne \chi_{0}} \left|\sum_{a \le A} \chi_{1}(a)\right|^{4} \ll A^{2}p(\log p)^{6}.
\end{align*}
Similarly,
\begin{equation*}
\sum_{\substack{\chi_{1} \ne \chi_{0} \\ \chi_{2} \ne \chi_{0} \\ \chi_{1}^{4}\chi_{2}^{6}=\chi_{0}}} |\mathcal{B}(\chi_{2})|^{4} \ll B^{2}p(\log p)^{6}.
\end{equation*}
Also,
\begin{align}\label{eq:4.8}
\sum_{\chi_{1},\chi_{2}} \bigl|\mathcal{W}_{p,d}(\chi_{1},\chi_{2})\bigr|^{2} &= \sum_{\chi_{1},\chi_{2}}\sum_{\substack{1 \le s,t < p \\ E_{s,t}(\mathbb{F}_{p})[d] \cong (\mathbb{Z}/d\mathbb{Z})^{2}}} \chi_{1}(s)\chi_{2}(t)\sum_{\substack{1 \le s^{\prime},t^{\prime} < p \\ E_{s^{\prime},t^{\prime}}(\mathbb{F}_{p})[d] \cong (\mathbb{Z}/d\mathbb{Z})^{2}}} \overline{\chi_{1}}(s^{\prime})\overline{\chi_{2}}(t^{\prime})\notag\\
&=\sum_{\substack{1 \le s,t < p \\ E_{s,t}(\mathbb{F}_{p})[d] \cong (\mathbb{Z}/d\mathbb{Z})^{2}}} \sum_{\substack{1 \le s^{\prime},t^{\prime} < p \\ E_{s^{\prime},t^{\prime}}(\mathbb{F}_{p})[d] \cong (\mathbb{Z}/d\mathbb{Z})^{2}}} \sum_{\chi_{1}} \chi_{1}(s)\overline{\chi_{1}}(s^{\prime})\sum_{\chi_{2}} \chi_{2}(t)\overline{\chi_{2}}(t^{\prime}) \notag\\
&=(p-1)^{2} \sum_{\substack{1 \le s,t < p \\ E_{s,t}(\mathbb{F}_{p})[d] \cong (\mathbb{Z}/d\mathbb{Z})^{2}}} 1 \notag\\
&\ll \frac{p^{4}}{d\psi(d)\varphi(d)}+p^{7/2}
\end{align}
by Lemma \ref{lemma:1}.  Putting all this information together, we obtain
\begin{equation*}
\left|\sum_{\substack{\chi_{1} \ne \chi_{0} \\ \chi_{2} \ne \chi_{0} \\ \chi_{1}^{4}\chi_{2}^{6}=\chi_{0}}} \mathcal{A}(\overline{\chi_{1}})\mathcal{B}(\overline{\chi_{2}})\mathcal{W}_{p,d}(\chi_{1},\chi_{2})\right|^{4} \ll \frac{(AB)^{2}p^{10}(\log p)^{12}}{d^{2}\psi(d)^{2}\varphi(d)^{2}}+\frac{(AB)^{2}p^{19/2}(\log p)^{12}}{d\psi(d)\varphi(d)}+(AB)^{2}p^{9}(\log p)^{12}.
\end{equation*}
Hence,
\begin{align*}
\Sigma_{4} &\ll \frac{1}{|\mathcal{C}|} \sum_{d \le \sqrt{x}+1} |g(d)| \sum_{\substack{p \le x \\ p \equiv 1 \bmod{d}}} \sqrt{AB}(\log p)^{3}\left(\frac{p^{1/2}}{d^{1/2}\psi(d)^{1/2}\varphi(d)^{1/2}}+\frac{p^{3/8}}{d^{1/4}\psi(d)^{1/4}\varphi(d)^{1/4}}+p^{1/4}\right) \\
&\ll \frac{1}{\sqrt{AB}}\left(x^{\frac{3}{2}}(\log x)^{2}+x^{1+\frac{\beta}{2}}(\log x)^{\gamma+3}(\log \log x)^{\frac{5}{4}}   +x^{\frac{5+2\beta}{4}}(\log x)^{\gamma+3} \log\log{x}  \right),
\end{align*}
as $\beta < 3/2$.  This completes the proof.
\end{proof}
\subsection{Proof of Theorem \ref{theorem1} }\begin{proof}
By combining \eqref{mterm}, \eqref{Error1}, and Lemma \ref{lemma:2}, we have
\begin{equation*}
\frac{1}{|\mathcal{C}|}\sum_{E_{a, b}\in \mathcal{C}} \sum_{p\leq x} f(i_{E_{a, b}}(p)) = \left(\sum_{d \ge 1} \frac{g(d)}{d\psi(d)\varphi(d)^2}\right)
 {\rm li}(x)+E,
 \end{equation*}
 where 
 \begin{align*}
E& \ll \frac{x}{(\log{x})^{c}}+\left(\frac{1}{A}+\frac{1}{B}\right)\left(\frac{x}{\log{x}}+x^{\frac{1+\beta}{2}}(\log x)^{\gamma+2}\right)+\left(\frac{1}{A^{1/r}}+\frac{1}{B^{1/r}}\right)x^{\frac{1+\beta}{2}+\frac{r+1}{4r^{2}}}(\log x)^{\gamma+1}\log \log x\\
&\qquad \qquad+\frac{1}{\sqrt{AB}}\left(x^{\frac{3}{2}}(\log x)^{2}+x^{1+\frac{\beta}{2}}(\log x)^{\gamma+3}(\log \log x)^{5/4} + x^{\frac{5+2\beta}{4}}(\log x)^{\gamma+3} \log\log{x}    \right),
\end{align*}
for given $c>1$ and $A,B>\exp\left( c_1 (\log{x})^{1/2} \right)$. Now
we choose $r$ large enough such that $\frac{1+\beta}{2}+\frac{r+1}{4r^2}<1$. (Note that we can do this if $\beta<1$.) So we arrive at the following upper bound for $E$. We have
\begin{align*}
E& \ll \frac{x}{(\log{x})^{c}}+x\exp\left(- \frac{c_1}{r} (\log{x})^{1/2} \right)  +\frac{1}{\sqrt{AB}}\left(x^{\frac{3}{2}}(\log x)^{2}+x^{\frac{5+2\beta}{4}}(\log x)^{\gamma+3}\log\log{x}\right).
\end{align*}
Now the result follows by choosing $AB\geq x(\log{x})^{4+2c}$ if $\beta<1/2$ and $AB\geq x^{1/2+\beta}(\log{x})^{2\gamma+6+2c}(\log\log{x})^2$ if $1/2\leq\beta<1$. 
\end{proof}

\subsection{Proof of Corollary \ref{corollaries}}\begin{proof}
Parts (i) and (ii) hold, since the characteristic function of $\{1\}$
can be written as
\begin{equation*}
\sum_{d|n} \mu(d)
\end{equation*}
and the divisor function can be written as
\begin{equation*}
\tau(n)=\sum_{d|n} 1.
\end{equation*}
Thus, $g(d)=\mu(d)$ and $g(d)=1$ both satisfy (\ref{bound}) with
$\beta=0$ and $\gamma=1$.\\
\indent For (iii), let $f(n)=1/n^{k}$, where $k \in \mathbb{N}$.  Then, writing
\begin{equation*}
f(n)=\sum_{d|n} g(d),
\end{equation*}
gives us that
\begin{equation*}
|g(n)|=\sum_{d|n} \left|\mu\left(\frac{n}{d}\right)f(d)\right| \le
\sum_{d|n} 1 \ll \tau(n).
\end{equation*}
Therefore, by Theorem \ref{theorem1}, we have
\begin{equation}\label{eq:cor}
\frac{1}{|\mathcal{C}|}\sum_{E \in \mathcal{C}} \sum_{p \le x}
\frac{1}{i_{E}(p)^{k}}=C_{k}\li(x)+O\left(\frac{x}{(\log
x)^{c}}\right).
\end{equation}
where $C_{k}$ is defined in the corollary.  Let $a_{p}(E)$ be defined by $\# E_{p}(\mathbb{F}_{p})=p+1-a_{p}(E)$.
Hasse's Theorem says that $|a_{p}(E)| \le 2\sqrt{p}$.  Note that
\begin{align*}
\sum_{E \in \mathcal{C}} \sum_{p \le x} e_{E}(p)^{k} &= \sum_{E \in
\mathcal{C}} \sum_{p \le x}
\left(\frac{p+1-a_{E}(p)}{i_{E}(p)}\right)^{k} = \sum_{E \in
\mathcal{C}} \sum_{p \le x}
\left(\frac{p^{k}}{i_{E}(p)^{k}}+\sum_{j=1}^{k} \binom{k}{j}
\frac{p^{k-j}(1-a_{p}(E))^{j}}{i_{E}(p)^{k}}\right) \\
&=\sum_{E \in \mathcal{C}} \sum_{p \le x}
\frac{p^{k}}{i_{E}(p)^{k}}+O_{k}\left(x^{k-\frac{1}{2}}\sum_{E \in
\mathcal{C}}\sum_{p \le x} \frac{1}{i_{p}(E)^{k}}\right)\\
&=\sum_{E \in \mathcal{C}} \sum_{p \le x}
\frac{p^{k}}{i_{E}(p)^{k}}+O_{k}\left(\frac{|\mathcal{C}|x^{k+\frac{1}{2}}}{\log{x}}\right).
\end{align*}
For the first part in the above, by (\ref{eq:cor}), we have
\begin{align*}
\frac{1}{|\mathcal{C}|}\sum_{E \in \mathcal{C}} \sum_{p \le x} \frac{p^{k}}{i_{E}(p)^{k}} 
&= C_{k}x^{k}\li(x)+O\left(\frac{x^{k+1}}{(\log
x)^{c}}\right)-C_{k}k\int_{2}^{x}t^{k-1}\li(t) \,\,
dt+O_{k}\left(\int_{2}^{x} \frac{t^{k}}{(\log t)^{c}} \,\, dt\right)\\
&=C_{k}x^{k}\li(x)-C_{k}k\int_{2}^{x}t^{k-1}\li(t) \,\,
dt+O\left(\frac{x^{k+1}}{(\log x)^{c}}\right).
\end{align*}
Then, the result holds since that there exists a constant $C$ such that
\begin{equation*}
\li(x^{k+1})+C=x^{k}\li(x)-k\int_{2}^{x}t^{k-1}\li(t) \,\, dt.
\end{equation*}
\end{proof}
%

\section{A TECHNICAL LEMMA}
\begin{lemma}\label{lemma:3}
Let $r \in \mathbb{N}$ and $\varepsilon > 0$ be fixed. Let $g:\mathbb{N} \to \mathbb{C}$ be a function such that
\begin{equation*}
\sum_{d \le x} |g(d)| \ll x^{1+\beta}(\log x)^{\gamma},
\end{equation*} 
where $0\leq \beta <3/4$ and $\gamma\in \mathbb{R}_{\geq 0}$. Then there are positive constants $c_1$ and $c_3$ such that if $A,B > \exp(c_{1}(\log x)^{1/2})$ we have
\begin{align*}
&\frac{4}{|\mathcal{C}|}\sum_{\substack{p,q \le x \\ p \ne q}} \frac{1}{(p-1)(q-1)} \sum_{\substack{s,t \in \mathbb{F}_{p}^{\times} \\ s^{\prime},t^{\prime} \in \mathbb{F}_{q}^{\times}}} \sum_{\substack{d|i_{E_{s,t}}(p) \\ d^{\prime}|i_{E_{s^{\prime},t^{\prime}}}(q)}} g(d)g(d^{\prime}) \left(\sum_{\substack{|a| \le A, |b| \le B: \\ \exists 1 \le u < p, 1 \le u^{\prime} < q \\ a \equiv su^{4} \bmod{p}, a \equiv s^{\prime}(u^{\prime})^{4} \bmod{q} \\ b \equiv tu^{6} \bmod{p}, b \equiv t^{\prime}(u^{\prime})^{6} \bmod{q}}} 1- \frac{AB}{pq}\right)\\
&\qquad \qquad \ll x(\log x)^{\gamma-1}(\log\log{x})+\left(\frac{1}{A}+\frac{1}{B}\right)\frac{x^2}{(\log{x})^2}+x^{2}\exp\left(-c_{3}\frac{(\log x)^{1/2}}{\log \log x}\right)
\\&\qquad \qquad \qquad +\left(\frac{1}{A^{1/r}}+\frac{1}{B^{1/r}}\right) x^{\frac{3+\beta}{2}+\frac{r+1}{2r^{2}}+2\varepsilon}(\log x)^{\gamma}\log\log{x}+\frac{1}{\sqrt{AB}}\left(x^{3}(\log x)+x^{\frac{11+2\beta}{4}}(\log x)^{2\gamma+3}(\log \log x)^2\right),
\end{align*}
where $c_{3}$ is a positive constant.
\end{lemma}
\begin{proof}
Throughout, a prime ${}^{\prime}$ superscript will denote that underlying object is related to the prime $q$.  Note that, for $p,q$ prime, $s,t \in \mathbb{F}_{p}^{\times}$ and $s^{\prime},t^{\prime} \in \mathbb{F}_{q}^{\times}$ fixed, by orthogonality relations, we have
\begin{align*}
\sum_{\substack{|a| \le A, |b| \le B: \\ \exists 1 \le u < p, 1 \le u^{\prime} < q \\ a \equiv su^{4} \bmod{p}, a \equiv s^{\prime}(u^{\prime})^{4} \bmod{q} \\ b \equiv tu^{6} \bmod{p}, b \equiv t^{\prime}(u^{\prime})^{6} \bmod{q}}} 1 &= \frac{1}{4} \sum_{1 \le u < p} \,\, \sum_{1 \le u^{\prime} < q} \,\, \sum_{|a| \le A} \,\, \sum_{|b| \le B} \left(\frac{1}{p-1} \sum_{\chi_{1} \bmod{p}} \chi_{1}(su^{4})\overline{\chi_{1}}(a)\right)\left(\frac{1}{p-1} \sum_{\chi_{2} \bmod{p}} \chi_{2}(tu^{6})\overline{\chi_{2}}(b)\right)\\
&\qquad \qquad \times \left(\frac{1}{q-1} \sum_{\chi_{1}^{\prime} \bmod{q}} \chi_{1}^{\prime}\left(s^{\prime}(u^{\prime})^{4}\right)\overline{\chi_{1}^{\prime}}(a)\right)\left(\frac{1}{q-1} \sum_{\chi_{2}^{\prime} \bmod{q}} \chi_{2}^{\prime}\left(t^{\prime}(u^{\prime})^{6}\right)\overline{\chi_{2}}(b)\right)\\
&=\frac{1}{4(p-1)(q-1)} \sum_{\substack{{\chi_{1}, \chi_{2} \bmod{p}}\\{\chi_1^4 \chi_2^6=\chi_{0}}}} \sum_{\substack{{\chi_{1}^{\prime}, \chi_{2}^{\prime} \bmod{q}}\\{(\chi_1^\prime)^4 (\chi_2^\prime)^6=\chi_{0}^{\prime}}}} \chi_{1}(s)\chi_{2}(t)\chi_{1}^{\prime}(s^{\prime})\chi_{2}^{\prime}(t^{\prime})\mathcal{A}(\overline{\chi_{1}\chi_{1}^{\prime}})\mathcal{B}(\overline{\chi_{2}\chi_{2}^{\prime}}),
\end{align*}
where
\begin{equation*}
\mathcal{A}(\chi):=\sum_{|a| \le A} \chi(a)
~~~~\qquad{\rm and}~~~~
\qquad\mathcal{B}(\chi):=\sum_{|b| \le B} \chi(b).
\end{equation*}
Thus,
\begin{equation*}
\sum_{\substack{|a| \le A, |b| \le B: \\ \exists 1 \le u < p, 1 \le u^{\prime} < q \\ a \equiv su^{4} \bmod{p}, a \equiv s^{\prime}(u^{\prime})^{4} \bmod{q} \\ b \equiv tu^{6} \bmod{p}, b \equiv t^{\prime}(u^{\prime})^{6} \bmod{q}}} 1=\sum_{j=1}^{16} S_j(p,q,s,t,s^{\prime},t^{\prime}),
\end{equation*}
where $S_{j}$ corresponds to one of the cases arising from choices of each of the following conditions:
\begin{equation*}
\left\{\substack{\displaystyle\chi_{1}=\chi_{0}, \chi_{2}=\chi_{0} \\ \displaystyle\chi_{1} = \chi_{0}, \chi_{2} \ne \chi_{0}: \chi_{2}^{6}=\chi_{0} \\ \displaystyle\chi_{1} \ne \chi_{0}, \chi_{2} = \chi_{0}: \chi_{1}^{4}=\chi_{0} \\ \displaystyle\chi_{1} \ne \chi_{0}, \chi_{2} \ne \chi_{0}: \chi_{1}^{4}\chi_{2}^{6}=\chi_{0}}\right\} \times \left\{\substack{\displaystyle\chi_{1}^{\prime}=\chi_{0}^{\prime}, \chi_{2}^{\prime}=\chi_{0}^{\prime} \\ \displaystyle\chi_{1}^{\prime} = \chi_{0}^{\prime}, \chi_{2}^{\prime} \ne \chi_{0}^{\prime}: \left(\chi_{2}^{\prime}\right)^{6}=\chi_{0}^{\prime} \\ \displaystyle\chi_{1}^{\prime} \ne \chi_{0}^{\prime}, \chi_{2}^{\prime} = \chi_{0}^{\prime}: \left(\chi_{1}^{\prime}\right)^{4}=\chi_{0}^{\prime} \\ \displaystyle\chi_{1}^{\prime} \ne \chi_{0}^{\prime}, \chi_{2}^{\prime} \ne \chi_{0}^{\prime}: \left(\chi_{1}^{\prime}\right)^{4}\left(\chi_{2}^{\prime}\right)^{6}=\chi_{0}^\prime}\right\}
\end{equation*}
From these 16 cases, there are essentially five different cases to handle.\\
\indent \underline{\textbf{Case 1}}:  all four of $\chi_{1},\chi_{2},\chi_{1}^{\prime},\chi_{2}^{\prime}$ are principal.\\
\indent Let this correspond to $j=1$.  Then, for $p \ne q$, we have
\begin{align*}
S_{1}(p,q,s,t,s^{\prime},t^{\prime})
&=\frac{AB}{pq}+O\left(\frac{AB}{p^{2}q}\right)+O\left(\frac{AB}{pq^{2}}\right)+O\left(\frac{A+B}{pq}\right).
\end{align*}
Thus, we have
\begin{align*}
&\frac{4}{|\mathcal{C}|}\sum_{\substack{p,q \le x \\ p \ne q}} \frac{1}{(p-1)(q-1)} \sum_{\substack{s,t \in \mathbb{F}_{p}^{\times} \\ s^{\prime},t^{\prime} \in \mathbb{F}_{q}^{\times}}} \sum_{\substack{d|i_{E_{s,t}}(p) \\ d^{\prime}|i_{E_{s^{\prime},t^{\prime}}}(q)}} g(d)g(d^{\prime}) \left(\sum_{\substack{|a| \le A, |b| \le B: \\ \exists 1 \le u < p, 1 \le u^{\prime} < q \\ a \equiv su^{4} \bmod{p}, a \equiv s^{\prime}(u^{\prime})^{4} \bmod{q} \\ b \equiv tu^{6} \bmod{p}, b \equiv t^{\prime}(u^{\prime})^{6} \bmod{q}}} 1- \frac{AB}{pq}\right)\\
&\qquad \qquad = \frac{4}{|\mathcal{C}|}\sum_{\substack{p,q \le x \\ p \ne q}} \frac{1}{(p-1)(q-1)} \sum_{\substack{s,t \in \mathbb{F}_{p}^{\times} \\ s^{\prime},t^{\prime} \in \mathbb{F}_{q}^{\times}}} \sum_{\substack{d|i_{E_{s,t}}(p) \\ d^{\prime}|i_{E_{s^{\prime},t^{\prime}}}(q)}} g(d)g(d^{\prime}) \left(\sum_{j=2}^{16} S(p,q,s,t,s^{\prime},t^{\prime})+O\left(\frac{AB}{p^{2}q}+\frac{AB}{pq^{2}}+\frac{A+B}{pq}\right)\right).
\end{align*}
The sums corresponding to $j=2,3,\dotsc,16$ are dealt with in Cases 2, 3, 4, and 5.  Here, we will bound the sums corresponding to the error terms above.  We have
\begin{align*}
&\frac{4}{|\mathcal{C}|}\sum_{\substack{p,q \le x \\ p \ne q}} \frac{1}{(p-1)(q-1)} \sum_{\substack{s,t \in \mathbb{F}_{p}^{\times} \\ s^{\prime},t^{\prime} \in \mathbb{F}_{q}^{\times}}} \sum_{\substack{d|i_{E_{s,t}}(p) \\ d^{\prime}|i_{E_{s^{\prime},t^{\prime}}}(q)}} g(d)g(d^{\prime}) \frac{AB}{p^{2}q} \\
&\qquad \ll \left(\sum_{p \le x} \frac{1}{p^{3}} \sum_{s,t \in \mathbb{F}_{p}^{\times}}\sum_{d|i_{E_{s,t}}(p)} |g(d)|\right)\left(\sum_{q \le x} \frac{1}{q^{2}} \sum_{s^{\prime},t^{\prime} \in \mathbb{F}_{q}^{\times}} \sum_{d|i_{E_{s^{\prime},t^{\prime}}}(q)}|g(d^{\prime})|\right).
\end{align*}
The first summation can be bounded as we bound $\Sigma_{1,1}$ in Subsection \ref{sigma1}, and the second summation can be bounded as we bound $\Sigma_{1,2}$ in Subsection \ref{sigma1}.  That is, by (\ref{eq:4.1}), (\ref{eq:4.2}), and $\beta<3/4$, we have
\begin{align*}
\frac{4}{|\mathcal{C}|}\sum_{p,q \le x} \frac{1}{(p-1)(q-1)} \sum_{\substack{s,t \in \mathbb{F}_{p}^{\times} \\ s^{\prime},t^{\prime} \in \mathbb{F}_{q}^{\times}}} \sum_{\substack{d|i_{E_{s,t}}(p) \\ d^{\prime}|i_{E_{s^{\prime},t^{\prime}}}(q)}} g(d)g(d^{\prime}) \frac{AB}{p^{2}q} \ll x(\log{x})^{\gamma-1} \log\log{x}.
\end{align*}
The same bound holds for the term coming from $O(AB/pq^2)$.
For the last error term, by (\ref{eq:4.2}), we have
\begin{align*}
&\frac{4}{|\mathcal{C}|}\sum_{\substack{p,q \le x \\ p \ne q}} \frac{1}{(p-1)(q-1)} \sum_{\substack{s,t \in \mathbb{F}_{p}^{\times} \\ s^{\prime},t^{\prime} \in \mathbb{F}_{q}^{\times}}} \sum_{\substack{d|i_{E_{s,t}}(p) \\ d^{\prime}|i_{E_{s^{\prime},t^{\prime}}}(q)}} g(d)g(d^{\prime}) \frac{A+B}{pq} \\
&\qquad \ll \left(\frac{1}{A}+\frac{1}{B}\right) \left(\sum_{p \le x} \frac{1}{p^{2}} \sum_{s,t \in \mathbb{F}_{p}^{\times}}\sum_{d|i_{E_{s,t}}(p)} |g(d)|\right)\left(\sum_{q \le x} \frac{1}{q^{2}} \sum_{s^{\prime},t^{\prime} \in \mathbb{F}_{q}^{\times}} \sum_{d|i_{E_{s^{\prime},t^{\prime}}}(q)}|g(d^{\prime})|\right) \\
&\qquad \ll \left(\frac{1}{A}+\frac{1}{B}\right) \frac{x^2}{(\log{x})^2}.
\end{align*}
\\
\indent \underline{\textbf{Case 2}}:  Exactly two of $\chi_1$, $\chi_2$, $\chi_1^\prime$, $\chi_2^\prime$ are principal. We have two subcases to consider.

\indent \underline{\textbf{Subcase 1}}: Exactly one of $\chi_{1}$ or $\chi_2$ is principal and exactly one of  $\chi_{1}^{\prime}$ or $\chi_2^{\prime}$ is principal. We will bound the summation when $\chi_1=\chi_0$ and $\chi_1^\prime=\chi_0^\prime$.  The bound for when $\chi_{1}=\chi_{0}$ and $\chi_{2}^{\prime}=\chi_{0}^{\prime}$ is similar.\\
\indent The estimation is analogous to estimations of $\Sigma_2$ and $\Sigma_3$ in Subsection \ref{subsubsec:3.4.2}. We note that $\chi_{0}\chi_{0}^{\prime}$ is the principal character modulo $pq$ since $p \ne q$.  Hence, $|\mathcal{A}(\chi_{0}\chi_{0}^{\prime})| \ll A$. Thus,
\begin{align}
&\frac{4}{|\mathcal{C}|}\sum_{\substack{p,q \le x \\ p \ne q}} \frac{1}{(p-1)(q-1)} \sum_{\substack{s,t \in \mathbb{F}_{p}^{\times} \\ s^{\prime},t^{\prime} \in \mathbb{F}_{q}^{\times}}} \sum_{\substack{d|i_{E_{s,t}}(p) \\ d^{\prime}|i_{E_{s^{\prime},t^{\prime}}}(q)}} g(d)g(d^{\prime}) \frac{1}{4(p-1)(q-1)}\sum_{\substack{\chi_{2} \ne \chi_{0},~ \chi_2^\prime\neq \chi_0^\prime \\ \chi_{2}^{6}=\chi_{0},~ (\chi_{2}^\prime)^{6}=\chi_{0}^\prime}} \chi_{2}(t)\chi_2^\prime(t^\prime)\mathcal{A}(\overline{\chi_{0}\chi_{0}^{\prime}})\mathcal{B}(\overline{\chi_{2}\chi_{2}^{\prime}})\nonumber\\ 
&\qquad \qquad \ll \frac{1}{B} \sum_{\substack{p,q \le x \\ p \ne q}} \frac{1}{p^{2}q^{2}}\sum_{\substack{d|p-1 \\ d \le \sqrt{p}+1 \\ d^{\prime}|q-1 \\ d \le \sqrt{q}+1}}|g(d)| \cdot |g(d^{\prime})|\sum_{\substack{s,t \in \mathbb{F}_{p}^{\times} \\ E_{s,t}(\mathbb{F}_{p})[d] \cong (\mathbb{Z}/d\mathbb{Z})^{2} \\ s^{\prime},t^{\prime} \in \mathbb{F}_{p}^{\times} \\ E_{s^{\prime},t^{\prime}}(\mathbb{F}_{q})[d^{\prime}] \cong (\mathbb{Z}/d^{\prime}\mathbb{Z})^{2}}}\sum_{\substack{\chi_{2} \ne \chi_{0},~ \chi_2^\prime\neq \chi_0^\prime \\ \chi_{2}^{6}=\chi_{0},~ (\chi_{2}^\prime)^{6}=\chi_{0}^\prime}}|\mathcal{B}(\overline{\chi_{2}\chi_{2}^{\prime}})| \nonumber \\
&\qquad \qquad \ll \frac{1}{B} \sum_{d \le \sqrt{x}+1} |g(d)| \sum_{d^{\prime} \le \sqrt{x}+1} |g(d^{\prime})| \sum_{\substack{p \le x \\ p \equiv 1 \bmod{d}}} \frac{1}{p^{2}} \sum_{\substack{q \le x \\ q \equiv 1 \bmod{d^{\prime}}}} \frac{1}{q^{2}} \sum_{\substack{\chi_{2} \ne \chi_{0},~ \chi_2^\prime\neq \chi_0^\prime \\ \chi_{2}^{6}=\chi_{0},~ (\chi_{2}^\prime)^{6}=\chi_{0}^\prime}} |\mathcal{B}(\overline{\chi_{2}\chi_{2}^{\prime}})| \nonumber \\
\label{sigma}
&\qquad \qquad \qquad \times \left(\frac{p(p-1)}{d\psi(d)\varphi(d)}+O(p^{3/2})\right)\left(\frac{q(q-1)}{d^{\prime}\psi(d^{\prime})\varphi(d^{\prime})}+O(q^{3/2})\right) \\
&\qquad \qquad =\sigma_1+\sigma_2+\sigma_3+\sigma_4, \nonumber
\end{align}
where $\sigma_1$ is the sum corresponding to the product of the main terms in \eqref{sigma}, $\sigma_4$ corresponds to the product of error terms in \eqref{sigma}, and $\sigma_2$ and $\sigma_3$ correspond to the mixed terms.
We will evaluate each of these summations separately.  For the first summation
we have
\begin{equation}\label{eq:4.32}
\sigma_{1} = \frac{1}{B}\sum_{d \le \sqrt{x}+1} \frac{|g(d)|}{d\psi(d)\varphi(d)} \sum_{d^\prime \le \sqrt{x}+1} \frac{|g(d^\prime)|}{d^\prime\psi(d^\prime)\varphi(d^\prime)}\sum_{\substack{p, q \le x \\ p \equiv 1 \bmod{d}\\ q \equiv 1 \bmod{d^\prime}}}\sum_{\substack{\chi_{2} \ne \chi_{0},~ \chi_2^\prime\neq \chi_0^\prime \\ \chi_{2}^{6}=\chi_{0},~ (\chi_{2}^\prime)^{6}=\chi_{0}^\prime}} |\mathcal{B}(\overline{\chi_{2}\chi_{2}^{\prime}})|  .
\end{equation}
Let $k=[4\log x/\log B]+1$.  By H\"{o}lder's inequality, we have
\begin{align}\label{eq:4.42}
\sum_{\substack{p, q \le x \\ p \equiv 1 \bmod{d}\\ q \equiv 1 \bmod{d^\prime}}} \sum_{\substack{\chi_{2} \ne \chi_{0},~ \chi_2^\prime\neq \chi_0^\prime \\ \chi_{2}^{6}=\chi_{0},~ (\chi_{2}^\prime)^{6}=\chi_{0}^\prime}} |\mathcal{B}(\overline{\chi_{2}\chi_2^\prime})| &\le \left(\sum_{\substack{p, q \le x \\ p \equiv 1 \bmod{d}\\ q \equiv 1 \bmod{d^\prime}}} \sum_{\substack{\chi_{2} \ne \chi_{0},~ \chi_2^\prime\neq \chi_0^\prime \\ \chi_{2}^{6}=\chi_{0},~ (\chi_{2}^\prime)^{6}=\chi_{0}^\prime}}1\right)^{1-\frac{1}{2k}}\left(\sum_{\substack{p, q \le x \\ p \equiv 1 \bmod{d}\\ q \equiv 1 \bmod{d^\prime}}}\sum_{\substack{\chi_{2} \ne \chi_{0},~ \chi_2^\prime\neq \chi_0^\prime \\ \chi_{2}^{6}=\chi_{0},~ (\chi_{2}^\prime)^{6}=\chi_{0}^\prime}} \left|\sum_{b \le B} \chi_{2}\chi_2^\prime(b)\right|^{2k}\right)^{\frac{1}{2k}} \notag \\
&\ll (\pi(x;d,1)\pi(x;d^\prime, 1))^{1-\frac{1}{2k}}\left(\sum_{p, q \le x}\sum_{\substack{\chi_2 \ne \chi_{0}, \chi_2^\prime\neq \chi_0^\prime}} \left|\sum_{b \le B^{k}} \tau_{k, B}(b)\chi_{2}\chi_2^\prime(b)\right|^{2}\right)^{\frac{1}{2k}},
\end{align}
where $\tau_{k, B}(n):=\#\left\{(a_{1},a_{2},\dotsc,a_{k}) \in [1,B]^{k} \cap \mathbb{N}^{k}:n=a_{1}a_{2}\dotsm a_{k}\right\}$.  By Lemma \ref{largesieve}, we have
\begin{equation}\label{eq:4.52}
\sum_{p, q \le x}\sum_{\substack{\chi \ne \chi_{0}}}^{\ast} \left|\sum_{b \le B^{k}} \tau_{k, B}(b)\chi(b)\right|^{2} \ll (x^{4}+B^{k})\sum_{b \le B^{k}} \tau_{k, B}(b)^{2}.
\end{equation}
Suppose $k=1$.  That is, $B > x^{4}$.  Then, we obtain
\begin{equation*}
\sum_{p,q \le x}\sum_{\substack{\chi_2 \ne \chi_{0}\\ \chi_2^\prime \ne \chi_0^\prime}} \left|\sum_{b \le B^{k}} \tau_{1,B}(b)\chi_{2}\chi_2^\prime(b)\right|^{2} \ll B^{2}.
\end{equation*}
Therefore by employing Lemma \ref{phi} (i) in  \eqref{eq:4.42}, we have
\begin{equation*}
\sum_{\substack{p, q \le x \\ p \equiv 1 \bmod{d}\\ q \equiv 1 \bmod{d^\prime}}} \sum_{\substack{\chi_{2} \ne \chi_{0},~ \chi_2^\prime\neq \chi_0^\prime \\ \chi_{2}^{6}=\chi_{0},~ (\chi_{2}^\prime)^{6}=\chi_{0}^\prime}} |\mathcal{B}(\overline{\chi_{2}\chi_2^\prime})| \ll B\frac{x}{\varphi(d)^{1/2}\varphi(d^\prime)^{1/2}(\log x)}.
\end{equation*}
Substituting this into Equation (\ref{eq:4.32}), we obtain
\begin{equation*}
\sigma_{1} \ll \frac{x}{\log x}\sum_{d \le x} \frac{|g(d)|}{d\psi(d)\varphi(d)^{3/2}} \sum_{d^\prime \le x} \frac{|g(d^\prime)|}{d^\prime\psi(d^\prime)\varphi(d^\prime)^{3/2}}\ll \frac{x}{\log x},
\end{equation*}
as $\beta < 3/4$. The latter summations were previously determined to be constants.\\
\indent Now suppose $k=[4\log{x}/\log{B}]+1 > 1$.  Then $B\leq x^4$ and $x^4 < B^k \leq Bx^4 \leq x^8$. 
Then, by Lemma \ref{Stephens} (i) and (iii), (\ref{eq:4.42}), (\ref{eq:4.52}), and the trivial bounds for $\pi(x;d,1)$ and $\pi(x;d^\prime, 1)$, we have
\begin{align}\label{eq:4.62}
\sum_{\substack{p \le x \\ p \equiv 1 \bmod{d}\\ q \equiv 1 \bmod{d^\prime}}} \sum_{\substack{\chi_{2} \ne \chi_{0},~ \chi_2^\prime\neq \chi_0^\prime \\ \chi_{2}^{6}=\chi_{0},~ (\chi_{2}^\prime)^{6}=\chi_{0}^\prime}} 
|\mathcal{B}(\overline{\chi_{2}\chi_{2}^{\prime}})| &\ll \left(\frac{x^2}{d d^\prime}\right)^{1-\frac{1}{2k}}\left((x^{4}+B^{k})B^{k}(\Psi(B,9\log x))^{k}\right)^{\frac{1}{2k}} \notag\\
&\ll B\frac{x^2}{(d d^\prime)^{3/4}} x^{-\frac{1}{k}}(\Psi(B,9\log x))^{1/2} \notag\\
&\ll B\frac{x^2}{(d d^\prime)^{3/4}}\exp\left(-c_{3}\frac{(\log x)^{1/2}}{\log \log x}\right),
\end{align}
where $c_{3}>0$ if $c_1$ is a suitable large constant. Substituting (\ref{eq:4.62}) into (\ref{eq:4.32}), we obtain
\begin{equation*}
\sigma_{1} \ll x^2\exp\left(-c_{3}\frac{(\log x)^{1/2}}{\log \log x}\right)\sum_{d \le x} \frac{|g(d)|}{d^{7/4}\psi(d)\varphi(d)} \sum_{d^\prime \le x} \frac{|g(d^\prime)|}{(d^\prime)^{7/4}\psi(d^\prime)\varphi(d^\prime)}\ll x^2\exp\left(-c_{3}\frac{(\log x)^{1/2}}{\log \log x}\right),
\end{equation*}
as $\beta<3/4$.

\indent By Lemma \ref{Burgess} (ii), for any $r \in \mathbb{N}$ and $\varepsilon > 0$, we have that our second summation $\sigma_2$ is bounded by
\begin{align*}
&\ll \frac{1}{B} \sum_{d \le \sqrt{x}+1} \frac{|g(d)|}{d\psi(d)\varphi(d)} \sum_{d^{\prime} \le \sqrt{x}+1} |g(d^{\prime})| \sum_{\substack{p \le x \\ p \equiv 1 \bmod{d}}} \sum_{\substack{q \le x \\ q \equiv 1 \bmod{d^{\prime}}}} \frac{1}{q^{1/2}}\sum_{\substack{\chi_{2} \ne \chi_{0}, ~\chi_2^\prime\ne \chi_0^\prime \\ \chi_{2}^{6}=\chi_{0}, ~(\chi_{2}^\prime)^{6}=\chi_{0}^\prime }} \left|\sum_{b \le B} \chi_{2}\chi_{2}^{\prime}(b)\right| \\
&\ll_{r,\varepsilon} \frac{1}{B}\sum_{d \le \sqrt{x}+1} \frac{|g(d)|}{d\psi(d)\varphi(d)} \sum_{d^{\prime} \le \sqrt{x}+1} |g(d^{\prime})| \sum_{\substack{p \le x \\ p \equiv 1 \bmod{d}}} \sum_{\substack{q \le x \\ q \equiv 1 \bmod{d^{\prime}}}} \frac{1}{q^{1/2}}\sum_{\substack{\chi_{2} \ne \chi_{0}, ~\chi_2^\prime\ne \chi_0^\prime \\ \chi_{2}^{6}=\chi_{0}, ~(\chi_{2}^\prime)^{6}=\chi_{0}^\prime }} B^{1-\frac{1}{r}}(pq)^{\frac{r+1}{4r^{2}}+\varepsilon} \\
&\ll \frac{x^{1+\frac{r+1}{4r^{2}}+\varepsilon}}{{B^{1/r}}\log x} \sum_{d \le \sqrt{x}+1} \frac{|g(d)|}{d\psi(d)\varphi(d)^{2}} \sum_{d^{\prime} \le \sqrt{x}+1} |g(d^{\prime})| \sum_{\substack{q \le x \\ q \equiv 1 \bmod{d^{\prime}}}} q^{\frac{-2r^{2}+r+1}{4r^{2}}+\varepsilon} \\
&\ll \frac{x^{\frac{3}{2}+\frac{r+1}{2r^{2}}+2\varepsilon}(\log\log{x})}{{B^{1/r}}(\log x)^2}\sum_{d^{\prime} \le \sqrt{x}+1} \frac{|g(d^{\prime})|}{d^{\prime}} \\
&\ll \frac{1}{B^{1/r}}x^{\frac{3+\beta}{2}+\frac{r+1}{2r^{2}}+2\varepsilon}(\log x)^{\gamma-1}\log\log{x}.
\end{align*}
In the above estimations we employed Lemma \ref{phi} (v) and the fact that $\beta < 3/4$.

\indent We obtain a similar bound for $\sigma_3$.\\
\indent Finally, by Lemma \ref{Burgess} (ii) and Lemma \ref{phi} (v), for any $r \in \mathbb{N}$ and $\varepsilon > 0$, we have that our fourth summation $\sigma_4$ is bounded by
\begin{align*}
&\ll \frac{1}{B} \sum_{d \le \sqrt{x}+1} |g(d)| \sum_{d^{\prime} \le \sqrt{x}+1} |g(d^{\prime})| \sum_{\substack{p \le x \\ p \equiv 1 \bmod{d}}} \frac{1}{p^{1/2}}\sum_{\substack{q \le x \\ q \equiv 1 \bmod{d^{\prime}}}} \frac{1}{q^{1/2}}\sum_{\substack{\chi_{2} \ne \chi_{0}, ~\chi_2^\prime\ne \chi_0^\prime \\ \chi_{2}^{6}=\chi_{0}, ~(\chi_{2}^\prime)^{6}=\chi_{0}^\prime }}  \left|\sum_{b \le B} \chi_{2}\chi_{0}^{\prime}(b)\right| \\
&\ll \frac{1}{B^{1/r}} \sum_{d \le \sqrt{x}+1} |g(d)| \sum_{d^{\prime} \le \sqrt{x}+1} |g(d^{\prime})| \sum_{\substack{p \le x \\ p \equiv 1 \bmod{d}}} \sum_{\substack{q \le x \\ q \equiv 1 \bmod{d^{\prime}}}} (pq)^{\frac{-2r^{2}+r+1}{4r^{2}}+\varepsilon} \\
&\ll \frac{x^{1+\frac{r+1}{2r^{2}}+2\varepsilon}(\log \log x)^{2}}{B^{1/r}(\log x)^{2}}\sum_{d \le \sqrt{x}+1} \frac{|g(d)|}{d} \sum_{d^{\prime} \le \sqrt{x}+1} \frac{|g(d^{\prime})|}{d^{\prime}} \\
&\ll \frac{1}{B^{1/r}} x^{1+\beta+\frac{r+1}{2r^{2}}+2\varepsilon}(\log x)^{2\gamma}(\log \log x)^{2}.
\end{align*}
Adding the above bounds for $\sigma_1$, $\sigma_2$, $\sigma_3$, and $\sigma_4$ concludes Subcase 1 of Case 2.\\
\indent \underline{\textbf{Subcase 2}}:  Either both $\chi_1$ and $\chi_2$ are principal or both $\chi_{1}^{\prime}$ and $\chi_2^\prime$ are principal. Without loss of generality we assume that $\chi_{1}^{\prime}=\chi_{0}^{\prime}$ and $\chi_{2}^{\prime}=\chi_{0}^{\prime}$.\\
\indent We have
\begin{align}
&\frac{4}{|\mathcal{C}|} \sum_{\substack{p,q \le x \\ p \ne q}} \frac{1}{(p-1)(q-1)}\sum_{\substack{s,t \in \mathbb{F}_{p}^{\times} \\ s^{\prime},t^{\prime} \in \mathbb{F}_{q}^{\times}}} \sum_{\substack{d|i_{E_{s,t}}(p) \\ d^{\prime}|i_{E_{s^{\prime},t^{\prime}}}(q)}} g(d)g(d^{\prime}) \frac{1}{4(p-1)(q-1)} \sum_{\substack{\chi_{1} \ne \chi_{0} \\ \chi_{2} \ne \chi_{0} \\ \chi_{1}^{4}\chi_{2}^{6}=\chi_{0}}} \chi_{1}(s)\chi_{2}(t)\mathcal{A}(\overline{\chi_{1}\chi_{0}^{\prime}})\mathcal{B}(\overline{\chi_{2}\chi_{0}^{\prime}})\nonumber\\ \label{original}
&\qquad = \frac{1}{|\mathcal{C}|} \sum_{d \le \sqrt{x}+1} g(d) \sum_{d^{\prime} \le \sqrt{x}+1} g(d^{\prime}) \sum_{\substack{p,q \le x \\ p \ne q \\ p \equiv 1 \bmod{d} \\ q \equiv 1 \bmod{d^{\prime}}}} \frac{1}{(p-1)^{2}(q-1)^{2}} \sum_{\substack{\chi_{1} \ne \chi_{0} \\ \chi_{2} \ne \chi_{0} \\ \chi_{1}^{4}\chi_{2}^{6}=\chi_{0}}} \mathcal{A}(\overline{\chi_{1}\chi_{0}^{\prime}})\mathcal{B}(\overline{\chi_{2}\chi_{0}^{\prime}}) \mathcal{W}_{p,q}(\chi_{1},\chi_{2}),
\end{align}
where
\begin{equation*}
\mathcal{W}_{p,q}(\chi_{1},\chi_{2}):=\sum_{\substack{s,t \in \mathbb{F}_{p}^{\times} \\ E_{s,t}(\mathbb{F}_{p})[d] \cong (\mathbb{Z}/d\mathbb{Z})^{2}}} \quad \sum_{\substack{s^{\prime},t^{\prime} \in \mathbb{F}_{q}^{\times} \\ E_{s^{\prime},t^{\prime}}(\mathbb{F}_{q})[d^{\prime}] \cong (\mathbb{Z}/d^{\prime}\mathbb{Z})^{2}}} \chi_{1}(s)\chi_{2}(t).
\end{equation*}
By applying the Cauchy-Schwarz inequality twice, we obtain
\begin{equation*}
\left|\sum_{\substack{\chi_{1} \ne \chi_{0} \\ \chi_{2} \ne \chi_{0} \\ \chi_{1}^{4}\chi_{2}^{6}=\chi_{0}}} \mathcal{A}(\overline{\chi_{1}\chi_{0}^{\prime}})\mathcal{B}(\overline{\chi_{2}\chi_{0}^{\prime}}) \mathcal{W}_{p,q}(\chi_{1},\chi_{2})\right|^{4} \le \left(\sum_{\substack{\chi_{1} \ne \chi_{0} \\ \chi_{2} \ne \chi_{0} \\ \chi_{1}^{4}\chi_{2}^{6}=\chi_{0}}} |\mathcal{W}_{p,q}(\chi_{1},\chi_{2})|^{2}\right)^{2}\left(\sum_{\substack{\chi_{1} \ne \chi_{0} \\ \chi_{2} \ne \chi_{0} \\ \chi_{1}^{4}\chi_{2}^{6}=\chi_{0}}} |\mathcal{A}(\overline{\chi_{1}\chi_{0}^{\prime}})|^{4}\right)\left(\sum_{\substack{\chi_{1} \ne \chi_{0} \\ \chi_{2} \ne \chi_{0} \\ \chi_{1}^{4}\chi_{2}^{6}=\chi_{0}}} |\mathcal{B}(\overline{\chi_{2}\chi_{0}^{\prime}})|^{4}\right).
\end{equation*}
From Lemma \ref{FI} we have
\begin{equation*}
\sum_{\substack{\chi_{1} \ne \chi_{0} \\ \chi_{2} \ne \chi_{0} \\ \chi_{1}^{4}\chi_{2}^{6}=\chi_{0}}} |\mathcal{A}(\overline{\chi_{1}\chi_{0}^{\prime}})|^{4} \ll A^{2}pq(\log pq)^{6}
\end{equation*}
and
\begin{equation*}
\sum_{\substack{\chi_{1} \ne \chi_{0} \\ \chi_{2} \ne \chi_{0} \\ \chi_{1}^{4}\chi_{2}^{6}=\chi_{0}}} |\mathcal{B}(\overline{\chi_{2}\chi_{0}^{\prime}})|^{4} \ll B^{2}pq(\log pq)^{6}.
\end{equation*}
We have
\begin{align*}
\sum_{\substack{\chi_{1} \ne \chi_{0} \\ \chi_{2} \ne \chi_{0} \\ \chi_{1}^{4}\chi_{2}^{6}=\chi_{0}}} |\mathcal{W}_{p,q}(\chi_{1},\chi_{2})|^{2} &\le \sum_{\chi_{1},\chi_{2}} \mathcal{W}_{p,q}(\chi_{1},\chi_{2})\overline{\mathcal{W}_{p,q}(\chi_{1},\chi_{2})} \\
&=\sum_{\chi_{1},\chi_{2}} \sum_{\substack{s,t \in \mathbb{F}_{p}^{\times} \\ s^{\prime},t^{\prime} \in \mathbb{F}_{q}^{\times} \\ E_{s,t}(\mathbb{F}_{p})[d] \cong (\mathbb{Z}/d\mathbb{Z})^{2} \\ E_{s^{\prime},t^{\prime}}(\mathbb{F}_{q})[d^{\prime}] \cong (\mathbb{Z}/d^{\prime}\mathbb{Z})^{2}}}\chi_{1}(s)\chi_{2}(t)\sum_{\substack{u,v \in \mathbb{F}_{p}^{\times} \\ u^{\prime},v^{\prime} \in \mathbb{F}_{q}^{\times} \\ E_{u,v}(\mathbb{F}_{p})[d] \cong (\mathbb{Z}/d\mathbb{Z})^{2} \\ E_{u^{\prime},v^{\prime}}(\mathbb{F}_{q})[d^{\prime}] \cong (\mathbb{Z}/d^{\prime}\mathbb{Z})^{2}}}\overline{\chi_{1}(u)}\,\overline{\chi_{2}(v)} \\
&=\sum_{\substack{s,t \in \mathbb{F}_{p}^{\times} \\ s^{\prime},t^{\prime} \in \mathbb{F}_{q}^{\times} \\ E_{s,t}(\mathbb{F}_{p})[d] \cong (\mathbb{Z}/d\mathbb{Z})^{2} \\ E_{s^{\prime},t^{\prime}}(\mathbb{F}_{q})[d^{\prime}] \cong (\mathbb{Z}/d^{\prime}\mathbb{Z})^{2}}}\sum_{\substack{u,v \in \mathbb{F}_{p}^{\times} \\ u^{\prime},v^{\prime} \in \mathbb{F}_{q}^{\times} \\ E_{u,v}(\mathbb{F}_{p})[d] \cong (\mathbb{Z}/d\mathbb{Z})^{2} \\ E_{u^{\prime},v^{\prime}}(\mathbb{F}_{q})[d^{\prime}] \cong (\mathbb{Z}/d^{\prime}\mathbb{Z})^{2}}} \sum_{\chi_{1}} \chi_{1}(s)\overline{\chi_{1}(u)} \sum_{\chi_{2}} \chi_{2}(t)\overline{\chi_{2}(v)} \\
&=\sum_{\substack{s,t \in \mathbb{F}_{p}^{\times} \\ s^{\prime},t^{\prime},u^{\prime},v^{\prime} \in \mathbb{F}_{q}^{\times} \\ E_{s,t}(\mathbb{F}_{p})[d] \cong (\mathbb{Z}/d\mathbb{Z})^{2} \\ E_{s^{\prime},t^{\prime}}(\mathbb{F}_{q})[d^{\prime}] \cong (\mathbb{Z}/d^{\prime}\mathbb{Z})^{2} \\ E_{u^{\prime},v^{\prime}}(\mathbb{F}_{q})[d^{\prime}] \cong (\mathbb{Z}/d^{\prime}\mathbb{Z})^{2}}} (p-1)(q-1) \ll pq\left(\frac{p^{2}}{d\psi(d)\varphi(d)}+p^{3/2}\right)\left(\frac{q^{4}}{(d^{\prime}\psi(d^{\prime})\varphi(d^{\prime}))^{2}}+q^{3}\right)\\
& \ll \frac{p^{3}q^{5}}{d(d^{\prime})^{2}\psi(d)\psi(d^{\prime})^{2}\varphi(d)\varphi(d^{\prime})^{2}}+\frac{p^{3}q^{4}}{d\psi(d)\varphi(d)}+\frac{p^{5/2}q^{5}}{(d^{\prime}\psi(d^{\prime})\varphi(d^{\prime}))^{2}}+p^{5/2}q^{4},
\end{align*}
which implies
\begin{align}
&\left|\sum_{\substack{\chi_{1} \ne \chi_{0} \\ \chi_{2} \ne \chi_{0} \\ \chi_{1}^{4}\chi_{2}^{6}=\chi_{0}}}  \mathcal{A}(\overline{\chi_{1}\chi_{0}^{\prime}})\mathcal{B}(\overline{\chi_{2}\chi_{0}^{\prime}}) \mathcal{W}_{p,q}(\chi_{1},\chi_{2})\right| \nonumber\\ \label{terms}
&\qquad \qquad \ll \sqrt{AB}(\log pq)^{3}\left(\frac{p^{2}q^{3}}{(d\psi(d)\varphi(d))^{1/2}d^{\prime}\psi(d^{\prime})\varphi(d^{\prime})}+\frac{p^{2}q^{5/2}}{(d\psi(d)\varphi(d))^{1/2}}+\frac{p^{7/4}q^{3}}{d^{\prime}\psi(d^{\prime})\varphi(d^{\prime})}+p^{7/4}q^{5/2}\right).
\end{align}
In the above inequalities, we have used the facts that $(a+b+c+d)^{2} \ll a^{2}+b^{2}+c^{2}+d^{2}$ and $(a+b+c+d)^{1/4} \ll a^{1/4}+b^{1/4}+c^{1/4}+d^{1/4}$, where the implied constants are absolute.\\
\indent Substituting the first term in  \eqref{terms} into the original summation in \eqref{original}, we obtain 
\begin{align}
& \frac{1}{\sqrt{AB}} \sum_{d \le \sqrt{x}+1} \frac{|g(d)|}{d^{1/2}\psi(d)^{1/2}\varphi(d)^{1/2}}\sum_{d^{\prime} \le \sqrt{x}+1} \frac{|g(d^{\prime})|}{d^{\prime}\psi(d^{\prime})\varphi(d^{\prime})} \sum_{\substack{p, q \le x \\ p \equiv 1 \bmod{d} \\ q \equiv 1 \bmod{d^{\prime}}}} q(\log pq)^{3} \nonumber\\
&\ll \frac{1}{\sqrt{AB}} x^{3}(\log x) \sum_{d \le \sqrt{x}+1} \frac{|g(d)|}{d^{1/2}\psi(d)^{1/2}\varphi(d)^{3/2}}\sum_{d^{\prime} \le \sqrt{x}+1} \frac{|g(d^{\prime})|}{(d^{\prime})^{}\psi(d^{\prime})\varphi(d^{\prime})^2} \nonumber\\
\label{first}
&\ll \frac{1}{\sqrt{AB}} x^{3}(\log x),
\end{align}
as $\beta < 3/4$.\\
\indent Similarly by substituting the second, third, and fourth terms in \eqref{terms} into the original summation in \eqref{original}, we obtain 
\begin{equation}
\label{second}
\frac{1}{\sqrt{AB}} \left(x^{(5+\beta)/2}(\log x)^{\gamma+2}(\log \log x)+ x^{(11+2\beta)/4}(\log x)^{\gamma+2}(\log \log x)+ x^{(9+4\beta)/4}(\log x)^{2\gamma+3}(\log \log x)^{2}\right).
\end{equation}
Adding \eqref{first} to \eqref{second} concludes Subcase 2 of Case 2.\\
\indent \underline{\textbf{Case 3}}:  Exactly three of $\chi_1$, $\chi_2$, $\chi_1^\prime$, and $\chi_2^\prime$ are principal. In this case by following the method of Subcase 1 of Case 2 we can conclude that the sum in question is bounded by the same bound in Subcase 1 of Case 2.\\
\indent \underline{\textbf{Case 4}}: Exactly one of $\chi_{1},\chi_{2},\chi_{1}^{\prime},\chi_{2}^{\prime}$ is principal.  
In this case by following the method of Subcase 2 of Case 2 we can conclude that the sum in question is bounded by the same bound in Subcase 2 of Case 2.

\indent \underline{\textbf{Case 5}}: All four of $\chi_{1},\chi_{2},\chi_{1}^{\prime},\chi_{2}^{\prime}$ are non-principal. In this case by following the method of Subcase 2 of Case 2 we can conclude that the sum in question is bounded by the same bound in Subcase 2 of Case 2. 

\end{proof}

\section{PROOF OF THEOREM \ref{Theorem 2}}
\begin{proof}
We will evaluate the following summation:
\begin{align}
\label{main-identity}
&\frac{1}{|\mathcal{C}|} \sum_{E \in \mathcal{C}} \left(\sum_{p \le x} f(i_{E}(p))-c_{0}(f)\li(x)\right)^{2}\nonumber \\
&\qquad \qquad=\frac{1}{|\mathcal{C}|}\sum_{E \in \mathcal{C}} \left(\sum_{\substack{p,q \le x \\ p \ne q}} f(i_{E}(p))f(i_{E}(q))+\sum_{p \le x} f(i_{E}(p))^{2}-2c_{0}(f)\li(x)\sum_{p \le x}f(i_{E}(p))+c_{0}(f)^{2}\li(x)^{2}\right).
\end{align}
For the first summation in \eqref{main-identity} we have 
\begin{align}
&\frac{1}{|\mathcal{C}|} \sum_{E \in \mathcal{C}} \sum_{\substack{p,q \le x \\ p \ne q}} f(i_{E}(p))f(i_{E}(q)) \nonumber\\
&\qquad \qquad= \frac{4}{|\mathcal{C}|}\sum_{\substack{p,q \le x \\ p \ne q}} \frac{1}{(p-1)(q-1)}\sum_{\substack{s,t \in \mathbb{F}_{p}^{\times} \\ s^{\prime},t^{\prime} \in \mathbb{F}_{q}^{\times}}} \sum_{\substack{d|i_{E_{s,t}}(p) \\ d^{\prime}|i_{E_{s^{\prime},t^{\prime}}}(q)}} g(d)g(d^{\prime})\sum_{\substack{|a| \le A, |b| \le B: \\ \exists 1 \le u < p, 1 \le u^{\prime} < q \\ a \equiv su^{4} \bmod{p}, a \equiv s^{\prime}(u^{\prime})^{4} \bmod{q} \\ b \equiv tu^{6} \bmod{p}, b \equiv t^{\prime}(u^{\prime})^{6} \bmod{q}}}1\nonumber\\
&\qquad \qquad \qquad +\frac{1}{|\mathcal{C}|}\sum_{\substack{p,q \le x \\ p \ne q}} \frac{|\aut_{\mathbb{F}_{p}}(E_{s,t})|\cdot|\aut_{\mathbb{F}_{q}}(E_{s^{\prime},t^{\prime}})|}{(p-1)(q-1)}\sum_{\substack{s,t \in \mathbb{F}_{p} \\ st=0 \\ s^{\prime},t^{\prime} \in \mathbb{F}_{q}^{\times}}} \sum_{\substack{d|i_{E_{s,t}}(p) \\ d^{\prime}|i_{E_{s^{\prime},t^{\prime}}}(q)}} g(d)g(d^{\prime})\sum_{\substack{|a| \le A, |b| \le B: \\ \exists 1 \le u < p, 1 \le u^{\prime} < q \\ a \equiv su^{4} \bmod{p}, a \equiv s^{\prime}(u^{\prime})^{4} \bmod{q} \\ b \equiv tu^{6} \bmod{p}, b \equiv t^{\prime}(u^{\prime})^{6} \bmod{q}}}1 \nonumber \\
&\qquad \qquad \qquad +\frac{1}{|\mathcal{C}|}\sum_{\substack{p,q \le x \\ p \ne q}} \frac{|\aut_{\mathbb{F}_{p}}(E_{s,t})|\cdot|\aut_{\mathbb{F}_{q}}(E_{s^{\prime},t^{\prime}})|}{(p-1)(q-1)}\sum_{\substack{s,t \in \mathbb{F}_{p}^{\times} \\ s^{\prime},t^{\prime} \in \mathbb{F}_{q} \\ s^{\prime}t^{\prime}=0}} \sum_{\substack{d|i_{E_{s,t}}(p) \\ d^{\prime}|i_{E_{s^{\prime},t^{\prime}}}(q)}} g(d)g(d^{\prime})\sum_{\substack{|a| \le A, |b| \le B: \\ \exists 1 \le u < p, 1 \le u^{\prime} < q \\ a \equiv su^{4} \bmod{p}, a \equiv s^{\prime}(u^{\prime})^{4} \bmod{q} \\ b \equiv tu^{6} \bmod{p}, b \equiv t^{\prime}(u^{\prime})^{6} \bmod{q}}}1 \nonumber\\
&\qquad \qquad \qquad +\frac{1}{|\mathcal{C}|}\sum_{\substack{p,q \le x \\ p \ne q}} \frac{|\aut_{\mathbb{F}_{p}}(E_{s,t})|\cdot|\aut_{\mathbb{F}_{q}}(E_{s^{\prime},t^{\prime}})|}{(p-1)(q-1)}\sum_{\substack{s,t \in \mathbb{F}_{p} \\ st=0 \\ s^{\prime},t^{\prime} \in \mathbb{F}_{q} \\ s^{\prime}t^{\prime}=0}} \sum_{\substack{d|i_{E_{s,t}}(p) \\ d^{\prime}|i_{E_{s^{\prime},t^{\prime}}}(q)}} g(d)g(d^{\prime})\sum_{\substack{|a| \le A, |b| \le B: \\ \exists 1 \le u < p, 1 \le u^{\prime} < q \\ a \equiv su^{4} \bmod{p}, a \equiv s^{\prime}(u^{\prime})^{4} \bmod{q} \\ b \equiv tu^{6} \bmod{p}, b \equiv t^{\prime}(u^{\prime})^{6} \bmod{q}}}1\nonumber\\
\label{sefr}
&\qquad \qquad = S_{1}+S_{2}+S_{3}+S_{4}.
\end{align}
Let $S$ be the corresponding bound in Lemma \ref{lemma:3} to a function $g(n)$ satisfying $$\sum_{n\leq x} |g(n)| \ll x^{1+\beta} (\log{x})^{\gamma+1}.$$

We have 
\begin{eqnarray}
S_{1} &=& O(S)+ \frac{4AB}{|\mathcal{C}|}\sum_{\substack{p,q \le x \\ p \ne q}} \frac{1}{p(p-1)q(q-1)}\sum_{\substack{s,t \in \mathbb{F}_{p}^{\times} \\ s^{\prime},t^{\prime} \in \mathbb{F}_{q}^{\times}}} f(i_{E_{s,t}}(p))f(i_{E_{s^\prime, t^\prime}}(q))\nonumber\\
\label{yek}
&=& O(S)+ \frac{4AB}{|\mathcal{C}|} \left (\left ( \sum_{p\leq x} \frac{1}{p(p-1)}  \sum_{\substack{s,t \in \mathbb{F}_{p}^{\times} }} f(i_{E_{s,t}}(p)) \right)^2-\sum_{p\leq x} \frac{1}{p^2(p-1)^2} \left( \sum_{\substack{s,t \in \mathbb{F}_{p}^{\times} }} f(i_{E_{s,t}}(p)) \right)^2\right).
\end{eqnarray}
From the calculation of the main term in Section \ref{themainterm} we have  
\begin{equation}
\label{do}
\sum_{p\leq x} \frac{1}{p(p-1)}  \sum_{\substack{s,t \in \mathbb{F}_{p}^{\times} }} f(i_{E_{s,t}}(p))=c_0(f) {\rm li}(x)+O\left( \frac{x}{(\log{x})^{c^\prime}}\right)
\end{equation}
for any $c^\prime >1$. Since $i_{E_{s, t}}(p) \leq \sqrt{p}+1$ and $f(n)\ll n^\beta (\log{n})^\gamma$, we have 
\begin{equation}
\label{se}
\sum_{p\leq x} \frac{1}{p^2(p-1)^2} \left( \sum_{\substack{s,t \in \mathbb{F}_{p}^{\times} }} f(i_{E_{s,t}}(p)) \right)^2 \ll x^{1+\beta} (\log{x})^{2\gamma-1}.
\end{equation}
As $\beta<3/4$, applying \eqref{yek} and \eqref{do} in \eqref{se}  yields
\begin{equation}
\label{s1}
S_1=c_0(f)^2 {\rm li}(x)^2+O(S)+O\left( \frac{x^2}{(\log{x})^{2 c^\prime}}\right)
\end{equation}
for any $c^\prime >1$.

We will next bound $S_{2}$ (a similar argument will bound $S_{3}$).  
We have
\begin{align}
S_{2} &\ll \frac{1}{|\mathcal{C}|}\sum_{\substack{p,q \le x \\ p \ne q}} \frac{|\aut_{\mathbb{F}_{p}}(E_{s,t})|\cdot|\aut_{\mathbb{F}_{q}}(E_{s^{\prime},t^{\prime}})|}{(p-1)(q-1)}\sum_{\substack{s,t \in \mathbb{F}_{p} \\ st=0 \\ s^{\prime},t^{\prime} \in \mathbb{F}_{q}^{\times}}} \sum_{\substack{d|i_{E_{s,t}}(p) \\ d^{\prime}|i_{E_{s^{\prime},t^{\prime}}}(q)}} |g(d)|\,|g(d^{\prime})|\sum_{\substack{|a| \le A, |b| \le B: \\ \exists 1 \le u < p, 1 \le u^{\prime} < q \\ a \equiv su^{4} \bmod{p}, a \equiv s^{\prime}(u^{\prime})^{4} \bmod{q} \\ b \equiv tu^{6} \bmod{p}, b \equiv t^{\prime}(u^{\prime})^{6} \bmod{q}}}1 \nonumber \\
&\ll \frac{1}{|\mathcal{C}|}\sum_{\substack{p,q \le x \\ p \ne q}} \frac{|\aut_{\mathbb{F}_{p}}(E_{s,t})|\cdot|\aut_{\mathbb{F}_{q}}(E_{s^{\prime},t^{\prime}})|}{(p-1)(q-1)}\sum_{\substack{s,t \in \mathbb{F}_{p} \\ st=0 \\ s^{\prime},t^{\prime} \in \mathbb{F}_{q}^{\times}}} \sum_{\substack{d|i_{E_{s,t}}(p) \\ d^{\prime}|i_{E_{s^{\prime},t^{\prime}}}(q)}} |g(d)|\,|g(d^{\prime})|\sum_{\substack{|a| \le A, |b| \le B: \\ \exists 1 \le u^{\prime} < q \\ a \equiv s^{\prime}(u^{\prime})^{4} \bmod{q} \\ b \equiv t^{\prime}(u^{\prime})^{6} \bmod{q}}}1 \nonumber\\
\label{char}
&\ll \left(\sum_{p \le x} \frac{1}{p}\sum_{\substack{s,t \in \mathbb{F}_{p} \\ st=0}} \sum_{d|i_{E_{s,t}}(p)} |g(d)|\right)\left(\frac{1}{|\mathcal{C}|}\sum_{q \le x} \sum_{s^{\prime},t^{\prime} \in \mathbb{F}_{q}^{\times}} \frac{|\aut_{\mathbb{F}_{q}}(E_{s^{\prime},t^{\prime}})|}{q-1} \sum_{d^{\prime}|i_{E_{s^{\prime},t^{\prime}}}(q)} |g(d^{\prime})| \sum_{\substack{|a| \le A, |b| \le B: \\ \exists 1 \le u^{\prime} < q \\ a \equiv s^{\prime}(u^{\prime})^{4} \bmod{q} \\ b \equiv t^{\prime}(u^{\prime})^{6} \bmod{q}}}1\right).
\end{align}
By Lemma \ref{claim} (iv), the first term in the above product is bounded by $x/\log{x}$.
The second term in the above product can be bounded by
\begin{align*}
&\ll \frac{1}{|\mathcal{C}|}\left(\sum_{q \le x} \frac{1}{q} \sum_{s^{\prime},t^{\prime} \in \mathbb{F}_{q}^{\times}} \sum_{d^{\prime}|i_{E_{s^{\prime},t^{\prime}}}(q)} |g(d^{\prime})| \left(\sum_{\substack{|a| \le A, |b| \le B: \\ \exists 1 \le u^{\prime} < q \\ a \equiv s^{\prime}(u^{\prime})^{4} \bmod^{\prime}{q} \\ b \equiv t^{\prime}(u^{\prime})^{6} \bmod^{\prime}{q}}}1-\frac{2AB}{q}\right)+\sum_{q \le x} \frac{1}{q} \sum_{s^{\prime},t^{\prime} \in \mathbb{F}_{q}^{\times}} \sum_{d^{\prime}|i_{E_{s^{\prime},t^{\prime}}}(q)} |g(d^{\prime})|\frac{2AB}{q}\right). \\
\end{align*}
Following the computations in Section \ref{themainterm} we can conclude that
$$\sum_{q \le x} \frac{1}{q} \sum_{s^{\prime},t^{\prime} \in \mathbb{F}_{q}^{\times}} \sum_{d^{\prime}|i_{E_{s^{\prime},t^{\prime}}}(q)} |g(d^{\prime})|\frac{2AB}{q}\ll AB \frac{x}{\log{x}}.$$
This together with Lemma \ref{lemma:2} imply that,  under the assumptions of Theorem \ref{Theorem 2}, the second term of the product in \eqref{char} is also bounded by $x/\log{x}$. Thus, we have
\begin{align}
\label{s2}
S_2&\ll
\frac{x^2}{(\log{x})^2}.
\end{align}
For $S_{4}$, we have
\begin{align*}
S_{4} 
&\ll \frac{1}{|\mathcal{C}|} \sum_{\substack{p,q \le x \\ p \ne q}} \frac{1}{pq} \sum_{\substack{s,t \in \mathbb{F}_{p} \\ st=0 \\ s^{\prime},t^{\prime} \in \mathbb{F}_{q} \\ s^{\prime}t^{\prime}=0}} \sum_{\substack{d|i_{E_{s,t}}(p) \\ d^{\prime}|i_{E_{s^{\prime},t^{\prime}}}(q)}} |g(d)| \cdot |g(d^{\prime})| \sum_{\substack{|a| \le A,|b| \le B \\ ab \equiv 0 \bmod{p} \\ ab \equiv 0 \bmod{q}}} 1.
\end{align*}
Note that
\begin{align*}
\sum_{\substack{|a| \le A, |b| \le B \\ ab \equiv 0 \bmod{pq}}} 1
&\ll \frac{AB}{pq}+O\left(A+B+\frac{B}{q}+\frac{B}{p}+\frac{B}{pq}\right) \ll \frac{AB}{pq}+O(A+B).
\end{align*}
Thus,
\begin{align}
\label{last}
S_4& \ll \frac{1}{|\mathcal{C}|} \sum_{\substack{p,q \le x \\ p \ne q}} \frac{1}{pq} \sum_{\substack{s,t \in \mathbb{F}_{p} \\ st=0 \\ s^{\prime},t^{\prime} \in \mathbb{F}_{q} \\ s^{\prime}t^{\prime}=0}} \sum_{\substack{d|i_{E_{s,t}}(p) \\ d^{\prime}|i_{E_{s^{\prime},t^{\prime}}}(q)}} |g(d)| \cdot |g(d^{\prime})| \left(\frac{AB}{pq}+A+B\right).
\end{align}
The summation in \eqref{last} corresponding to $AB/pq$ can be bounded by
\begin{align*}
&\left(\sum_{d \le \sqrt{x}+1} |g(d)| \sum_{\substack{p \le x \\ p \equiv 1 \bmod{d}}} \frac{1}{p}\right)^{2} \ll (\log \log x)^{2}(\log x)^{2}\left(\sum_{d \le \sqrt{x}+1} \frac{|g(d)|}{d}\right)^{2} 
\ll x^{\beta}(\log \log x)^{2}(\log x)^{2\gamma+4}.
\end{align*}
By employing Lemma \ref{claim} (iv), the summation in \eqref{last} corresponding to $A+B$ can be bounded by
\begin{align*}
&\ll \left(\frac{1}{A}+\frac{1}{B}\right)\left(\sum_{p \le x} \frac{1}{p} \sum_{\substack{s,t \in \mathbb{F}_{p} \\ st=0}} \sum_{d|i_{E_{s,t}}(p)} |g(d)|\right)^{2} \ll \left(\frac{1}{A}+\frac{1}{B}\right)\frac{x^2}{(\log{x})^2}.
\end{align*}
In conclusion we have
\begin{equation}
\label{s4}
S_4 \ll x^{\beta}(\log \log x)^{2}(\log x)^{2\gamma+4}+\left(\frac{1}{A}+\frac{1}{B}\right)\frac{x^2}{(\log{x})^2}.
\end{equation}
\noindent Thus,  under the assumptions of Theorem \ref{Theorem 2}, by applying \eqref{s1}, \eqref{s2}, and \eqref{s4}
in \eqref{sefr},  we have
\begin{equation}
\label{pang}
\frac{1}{|\mathcal{C}|} \sum_{E \in \mathcal{C}} \sum_{\substack{p,q \le x \\ p \ne q}} f(i_{E}(p))f(i_{E}(q)) =c_0(f)^2 {\rm li}(x)^2+O(S)+O\left( \frac{x^2}{(\log{x})^2}  \right).
\end{equation}


Next we bound $\sum_{p \le x} f(i_{E}(p))^{2}$.
Let $G:\mathbb{N} \to \mathbb{C}$ be defined by
\begin{equation*}
f(n)^{2}=\sum_{d|n} G(d).
\end{equation*}
Then, we have
\begin{equation*}
\sum_{n \le x} |G(n)| \le \sum_{d \le x} |f(d)|^{2} \sum_{\substack{n \le x \\ d|n}} 1 \le x\sum_{d \le x} \frac{|f(d)|^{2}}{d} \ll x^{1+2\beta}(\log x)^{2\gamma+1}.
\end{equation*}
Thus, applying the proof of Theorem \ref{theorem1} for $G$ and $f^2$ yields
\begin{align}
\frac{1}{|\mathcal{C}|}\sum_{E\in \mathcal{C}}\sum_{p \le x} f(i_{E}(p))^{2}& \ll \frac{x}{\log{x}}+\left(\frac{1}{A}+\frac{1}{B}\right)\left(\frac{x}{\log{x}}+x^{\frac{1+2\beta}{2}}(\log x)^{2\gamma+3}\right)+\left(\frac{1}{A^{1/r}}+\frac{1}{B^{1/r}}\right)x^{\frac{1+2\beta}{2}+\frac{r+1}{4r^{2}}}(\log x)^{2\gamma+2}\log \log x\nonumber\\
&\qquad \qquad+\frac{1}{\sqrt{AB}}\left(x^{\frac{3}{2}}(\log x)^{2}+x^{1+{\beta}}(\log x)^{2\gamma+4}(\log \log x)^{5/4} + x^{\frac{5+4\beta}{4}}(\log x)^{2\gamma+4} \log\log{x}    \right).\nonumber
\end{align}
Therefore
\begin{equation}
\label{f2}
\frac{1}{|\mathcal{C}|}\sum_{E\in \mathcal{C}}\sum_{p \le x} f(i_{E}(p))^{2}=O(S).
\end{equation}

Now by applying \eqref{pang} and \eqref{f2} to \eqref{main-identity} we conclude that, under the assumptions of Theorem \ref{Theorem 2},  we have
$$\frac{1}{|\mathcal{C}|} \sum_{E \in \mathcal{C}} \left(\sum_{p \le x} f(i_{E}(p))-c_{0}(f)\li(x)\right)^{2} =O(S)+O\left( \frac{x^2}{(\log{x})^2}\right).$$ 
Since $S=O(x^2/(\log{x})^2)$ the result follows.
\end{proof}

\subsection*{Acknowledgements}
The authors would like to thank Chantal David and Igor Shparlinski for correspondence on an earlier version of this paper.

\end{document}